\documentclass[11pt,a4paper, british]{amsart}

\usepackage{babel}
\usepackage[useregional]{datetime2}

\usepackage{amssymb}
\usepackage[hidelinks]{hyperref}

\usepackage{caption}
\usepackage{multirow}
\usepackage[table]{xcolor}

\newtheorem{thm}{Theorem}[section]
\newtheorem{prop}[thm]{Proposition}
\newtheorem{lem}[thm]{Lemma}

\newtheorem{claim}[thm]{Claim} 

\theoremstyle{definition}

\theoremstyle{remark}
\newtheorem{remark}[thm]{Remark}

\numberwithin{equation}{section}

\newcommand{\alg}{\overline{\mathbb{Q}}} 
\newcommand{\Q}{\mathbb{Q}} 
\newcommand{\Z}{\mathbb{Z}} 
\newcommand{\C}{\mathbb{C}} 
\newcommand{\h}{\mathbb{H}} 
\newcommand{\SL}{\mathrm{SL}_2(\Z)} 
\newcommand{\Gal}{\mathrm{Gal}} 

\DeclareMathOperator{\re}{Re} 
\DeclareMathOperator{\im}{Im} 
\DeclareMathOperator{\cl}{cl} 

\begin{document}
	
		\title[Andr\'e--Oort for sums of powers in $\C^n$]{Some uniform effective results on Andr\'e--Oort for sums of powers in $\C^n$}
	\author{Guy Fowler}
	\address{Department of Mathematics, University of Manchester, Manchester, UK
	\newline
	\indent
	Heilbronn Institute for Mathematical Research, Bristol, UK}
		\email{\href{mailto:guy.fowler@manchester.ac.uk}{guy.fowler@manchester.ac.uk}}
	\urladdr{\url{https://www.guyfowler.uk/}}
	\date{\today}

	\begin{abstract}
	We prove an Andr\'e--Oort-type result for a family of hypersurfaces in $\C^n$ that is both uniform and effective. Let $K_*$ denote the single exceptional imaginary quadratic field which occurs in the Siegel--Tatuzawa lower bound for the class number. We prove that, for $m, n \in \Z_{>0}$, there exists an effective constant $c(m, n)>0$ with the following property: if pairwise distinct singular moduli $x_1, \ldots, x_n$ with respective discriminants $\Delta_1, \ldots, \Delta_n$ are such that $a_1 x_1^m + \ldots + a_n x_n^m \in \Q$ for some $a_1, \ldots, a_n \in \Q \setminus \{0\}$ and $\# \{ \Delta_i : \Q(\sqrt{\Delta_i}) = K_*\}  \leq 1$, then $\max_i \lvert \Delta_i \rvert \leq c(m, n)$. In addition, we prove an unconditional and completely explicit version of this result when $(m, n) = (1, 3)$  and thereby determine all the triples $(x_1, x_2, x_3)$ of singular moduli such that $a_1 x_1 + a_2 x_2 + a_3 x_3 \in \Q$ for some $a_1, a_2, a_3 \in \Q \setminus \{0\}$.
\end{abstract}
	
	\maketitle
	
	
	\section{Introduction}\label{sec:intro}
	
	A singular modulus is the $j$-invariant of an elliptic curve with complex multiplication. The discriminant of a singular modulus is defined to be the discriminant of the imaginary quadratic order isomorphic to the endomorphism ring of the corresponding elliptic curve. In particular, the discriminant is a negative integer. There are only finitely many singular moduli of a given discriminant and these may be computed effectively \cite[\S13]{Cox22}.
	
	Identify $\C$ with the modular curve $Y(1)$ via the $j$-invariant. A point $(x_1, \ldots, x_n) \in \C^n$ such that $x_1, \ldots, x_n$ are all singular moduli is called, in the terminology of Shimura varieties, a special point of $\C^n$. A special point is a zero-dimensional special subvariety of $\C^n$ (see \cite[Definition~4.10]{Pila22} for the general definition of a special subvariety of $\C^n$). 
	
	The Andr\'e--Oort conjecture, which was proved for $\C^n$ by Pila \cite{Pila11}, states that a subvariety $V \subset \C^n$ contains only finitely many maximal special subvarieties. In particular, a subvariety $V$ contains only finitely many special points which do not lie on the union of all the positive-dimensional special subvarieties of $V$. Pila's proof of Andr\'e--Oort has a strong uniformity, as illustrated by the following theorem. This result is a direct consequence of Pila's uniform Andr\'e--Oort theorem \cite[Theorem~13.2]{Pila11} and a result of Binyamini \cite[Corollary~4]{Binyamini19}. The result is ineffective, due to the ineffectivity of Pila's proof of Andr\'e--Oort (see \cite[\S13]{Pila11}).
	
	\begin{thm}\label{thm:ineffpowers}
		Let $m, n, d \in \Z_{>0}$. There exists an ineffective constant $c(m, n, d) > 0$ with the following property:
		 
		 Let $x_1, \ldots, x_n$ be pairwise distinct singular moduli and write $\Delta_i$ for the discriminant of $x_i$. If
		\[ a_1 x_1^m + \ldots + a_n x_n^m = b\]
		for some $a_1, \ldots, a_n \in \alg \setminus \{0\}$ and $b \in \alg$ with
		\[ [\Q(a_1, \ldots, a_n, b) : \Q] \leq d,\]
		then
		\[ \max \{ \lvert \Delta_i \rvert : i=1, \ldots, n\} \leq c(m, n, d).\]
	\end{thm}

Notably, the constant $c(m, n, d)$ in Theorem~\ref{thm:ineffpowers} does not depend on the height of the coefficients $a_1, \ldots, a_n, b$. In particular, given $m, n \in \Z_{>0}$, there are only finitely many $n$-tuples $(x_1, \ldots, x_n)$ of pairwise distinct singular moduli $x_1, \ldots, x_n$ such that there exist $a_1, \ldots, a_n \in \Q \setminus \{0\}$ and $b \in \Q$ with
\[ a_1 x_1^m + \ldots + a_n x_n^m = b.\]
	
 Theorem~\ref{thm:ineffpowers} is known effectively only when $n \leq 2$. For $n = 1$, this is a consequence of a result due to Goldfeld \cite{Goldfeld76} and Gross and Zagier \cite{GrossZagier86} (see Proposition~\ref{prop:class}). For $n = 2$, an effective version of Theorem~\ref{thm:ineffpowers} follows from a theorem of K\"uhne \cite[Theorem~4]{Kuhne13} combined with the aforementioned result of Goldfeld--Gross--Zagier.
 
 \subsection{Main results}
	
		The first main result of this paper is the following effective partial version of Theorem~\ref{thm:ineffpowers} in the case that $d=1$, i.e.~for equations over $\Q$. Throughout this paper, $K_*$ denotes a fixed imaginary quadratic field, the definition of which is given just before Proposition~\ref{prop:Tatbd}. Briefly, $K_*$ is the single exceptional field arising from an application of Tatuzawa's effective version \cite{Tatuzawa51} of Siegel's lower bound \cite{Siegel35} for the class number of an imaginary quadratic field. 
		
		\begin{thm}\label{thm:effpowers}
			Let $m, n \in \Z_{>0}$. There exists an effective constant $c(m, n)$ with the following property: 
			
			Let $x_1, \ldots, x_n$ be pairwise distinct singular moduli and write $\Delta_i$ for the discriminant of $x_i$. If
			\[ a_1 x_1^m + \ldots + a_n x_n^m = b\]
			for some $a_1, \ldots, a_n \in \Q \setminus \{0\}$ and $b \in \Q$ and
			\[  \# \left \{ \Delta_i :  \, i \in \left \{1, \ldots, n\right \} \mbox{ such that } \Q\left(\sqrt{\Delta_i}\right) = K_* \right \} \leq 1,\]
			then
			\[ \max \{ \lvert \Delta_i \rvert : i=1, \ldots, n\} \leq c(m, n).\]
		\end{thm}

	An effective bound, depending only on $m, n$, for all those discriminants $\Delta_i$ such that $\Q(\sqrt{\Delta_i}) \neq K_*$ is given by a result of Binyamini \cite[Theorem~1]{Binyamini19}. If there is at most one $i$ such that $\Q(\sqrt{\Delta_i}) = K_*$, then $\max \{ \lvert \Delta_1 \rvert, \ldots, \lvert \Delta_n \rvert\}$ is effectively bounded in terms of $m, n$ by \cite[Corollary 1]{Binyamini19}. Our Theorem~\ref{thm:effpowers} improves on these prior results by allowing any number of the discriminants $\Delta_1, \ldots, \Delta_n$ to be such that $\Q(\sqrt{\Delta_i}) = K_*$, provided that all these exceptional $\Delta_i$ are themselves equal to one another.

The second main result of this paper shows that, under a more restrictive condition on the discriminants involved, we can make the constant $c(m, n)$ uniform in $m$ and also write down this constant explicitly.
	
	\begin{thm}\label{thm:unifeffpowers}
		Let $n \in \Z_{>0}$. There exist explicit constants $c_1(n), c_2(n)$ with the following properties: 
		
		Let $x_1, \ldots, x_n$ be pairwise distinct singular moduli such that
		\[ a_1 x_1^m + \ldots + a_n x_n^m = b\]
		for some $m \in \Z_{>0}$ and $a_1, \ldots, a_n \in \Q \setminus \{0\}$ and $b \in \Q$.
		Write $\Delta_i$ for the discriminant of $x_i$. 
		\begin{enumerate}
			\item 	If	$k \in \{1, \ldots, n\}$ is such that $\Q(\sqrt{\Delta_k}) \neq K_*$ and
			\[  \quad \# \left \{ \Delta_i :  \, i \in \left \{1, \ldots, n\right \} \mbox{ such that } \Q\left(\sqrt{\Delta_i}\right) = \Q\left(\sqrt{\Delta_k}\right) \right \} = 1,\]
		then $\lvert \Delta_k \rvert  \leq c_1(n)$.
			\item If, for every	$k \in \{1, \ldots, n\}$,
			\[ \quad  \# \left \{ \Delta_i :  \, i \in \left \{1, \ldots, n\right \} \mbox{ such that } \Q\left(\sqrt{\Delta_i}\right) = \Q\left(\sqrt{\Delta_k}\right) \right \} = 1,\]
			then $ \max \{ \lvert \Delta_i \rvert : i=1, \ldots, n\} \leq c_2(n)$.
		\end{enumerate}
	\end{thm}

An explicit form for the constant $c_1(n)$ is given in Proposition~\ref{prop:step4}, while an explicit form for the constant $c_2(n)$ appears at the end of Section~\ref{subsec:unipf}. 
	
	As steps in the proofs of Theorems~\ref{thm:effpowers} and Theorem~\ref{thm:unifeffpowers}, we also prove the following two results, which may have some independent interest. The first of them is also used in the proof of Theorem~\ref{thm:trip}.
	
	\begin{thm}\label{thm:equal}
Let $n \in \Z_{>0}$. Let $x_1, \ldots, x_n$ be  pairwise distinct singular moduli which are all of discriminant $\Delta$. Denote by $h(\Delta)$ the class number of the imaginary quadratic order of discriminant $\Delta$. Let $K = \Q(\sqrt{\Delta})$. If
\[a_1 x_1^m + \ldots + a_n x_n^m \in K\]
for some $a_1, \ldots, a_n \in K \setminus \{0\}$ and $m \in \Z_{>0}$, then either
\[ \lvert \Delta \rvert^{1/2} \leq \frac{1}{\pi} \left(\left(2n+3\right) \log\left(n+1\right) -2n +4\right),\]
or $h(\Delta) = n$ and $a_1 = \ldots = a_n$.
	\end{thm}
	
	The second is a result on the fields generated by linear combinations of powers of singular moduli of the same discriminant.
	
	\begin{thm}\label{thm:field}
Let $n \in \Z_{>0}$. Let $x_1, \ldots, x_n$ be  pairwise distinct singular moduli which are all of discriminant $\Delta$. Let $K = \Q(\sqrt{\Delta})$. Then either
\[ \lvert \Delta \rvert^{1/2} \leq \frac{1}{\pi} \left(\left(4n+3\right) \log\left(2n+1\right) -4n + 4\right),\]
or
\[ \left[K\left(x_1, a_1 x_1^m + \ldots + a_n x_n^m\right) : K\left(a_1 x_1^m + \ldots + a_n x_n^m\right) \right] \leq n\]
for all $a_1, \ldots, a_n \in K \setminus \{0\}$ and every $m \in \Z_{>0}$.
	\end{thm}
	
	We emphasise that Theorems~\ref{thm:equal} and~\ref{thm:field} apply to all discriminants $\Delta$, including those with $\Q(\sqrt{\Delta}) = K_*$. Note also that Theorems~\ref{thm:equal} and \ref{thm:field} are uniform in the exponent $m$, as well as in $a_1, \ldots, a_n$. For discussion of whether such uniformity in $m$ may also hold in Theorem~\ref{thm:ineffpowers}, see Section~\ref{subsec:dep}.
	
	The final main result of this paper is a completely explicit version of Theorem~\ref{thm:ineffpowers} for the case where $m= d=1$ and $n=3$. The analogous result for $n=2$ is due to Allombert, Bilu, and Pizarro-Madariaga \cite[Theorem~1.2]{AllombertBiluMadariaga15}. 
	
	\begin{thm}\label{thm:trip}
		Let $x, y, z$ be pairwise distinct singular moduli and $A, B, C \in \Q \setminus \{0\}$. Then
		\[ Ax + By + Cz \in \Q\]
		if and only if (up to permuting $x, y, z$) one of the following holds:
		\begin{enumerate}
			\item $x, y, z \in \Q$;
			\item	\begin{enumerate}
				\item $x \in \Q$,
				\item $\Q(y) = \Q(z)$,
				\item $[\Q(y) : \Q] = [\Q(z) : \Q] = 2$, and
				\item $B/C = -(z-z')/(y-y')$, where $y', z'$ are the unique non-trivial Galois conjugates of $y, z$ over $\Q$ respectively;
			\end{enumerate}
			\item 
			\begin{enumerate}
				\item $\Q(x) = \Q(y) = \Q(z)$, 
				\item $[\Q(x) : \Q] = [\Q(y) : \Q] = [\Q(z) : \Q] = 2$, and
				\item writing $x', y', z'$ for the unique non-trivial Galois conjugates over $\Q$ of $x, y, z$ respectively, we have that
				\[A = -\frac{B(y-y')+C(z-z')}{x-x'};\]
			\end{enumerate}
			\item 
			\begin{enumerate}
				\item $[\Q(x) : \Q] = [\Q(y) : \Q] = [\Q(z) : \Q] = 3$,
				\item $x, y, z$ are all conjugate over $\Q$, and
				\item $A= B = C$;
			\end{enumerate}
			\item \begin{enumerate}
				\item $\Q(x) \subset \Q(y) = \Q(z)$,
				\item $[\Q(x) : \Q] = 2$ and $[\Q(y) : \Q] = [\Q(z) : \Q] = 4$,
				\item $y, z$ are conjugate over $\Q$, and
				\item writing $x'$ for the unique non-trivial Galois conjugate of $x$ over $\Q$ and $v, w$ for the other two Galois conjugates of $y, z$ over $\Q$, we have that
				\[\frac{A}{B} = \frac{A}{C} = - \frac{(y+z) - (v+w)}{x - x'}.\]	
			\end{enumerate}
		\end{enumerate}
	\end{thm}

Note that it is straightforward to compute the list of all the triples of singular moduli satisfying one of the conditions (1)--(5) in Theorem~\ref{thm:trip}. The proof of Theorem~\ref{thm:trip} involves some computations in PARI \cite{PARI24}. These computations were carried out using a standard desktop computer\footnote{With an Intel Core i5-10600 processor and 16 GB RAM.}; the scripts are available from: \url{https://github.com/guyfowler/sums_of_powers}. 
	
	\subsection{Related results}
	
		Riffaut \cite{Riffaut19} and, jointly, Luca and Riffaut \cite{LucaRiffaut19} proved an effective (indeed, completely explicit) version \cite[Theorem~1.3]{LucaRiffaut19} of Theorem~\ref{thm:ineffpowers} in the case where $d = 1$ and $n = 2$. For $d=1$ and $n = 3$, an explicit version of Theorem~\ref{thm:ineffpowers} in the special case that $\lvert a_1 \rvert = \lvert a_2 \rvert = \lvert a_3 \rvert$ was proved by the author in a previous paper \cite[Theorem~1.1]{Fowler23}. 
	
	Binyamini \cite{Binyamini19} proved a version of Theorem~\ref{thm:ineffpowers} which is effective, but not uniform in the height of the coefficients $a_1, \ldots, a_n, b$. He proved \cite[Corollary~4]{Binyamini19} that if $m, n \in \Z_{>0}$ and $x_1, \ldots, x_n$ are pairwise distinct singular moduli of respective discriminants $\Delta_1, \ldots, \Delta_n$ such that 
\begin{align}\label{eq:lin}
	a_1 x_1^m + \ldots + a_n x_n^m = b \mbox{ for some }
	a_1, \ldots, a_n \in \alg \setminus \{0\} \mbox{ and } b \in \alg,
\end{align}
then
\[ \max \left \{\left\lvert \Delta_1 \right\rvert, \ldots, \left\lvert \Delta_n \right\rvert \right\}\leq c(m, n, d, h),\]
where $c(m, n, d, h)$ is an effective constant which depends only on $m, n,$ 
\[ d = [\Q(a_1, \ldots, a_n, b) : \Q],\]
and also
\[ h = H((a_1, \ldots, a_n, b)).\]
Here $H(\cdot)$ denotes the absolute multiplicative Weil height, see e.g. \cite[\S1.5]{BombieriGubler06}. In the $m=1$ case, the same result was proved independently by Bilu and K\"uhne \cite[Lemma~3.1]{BiluKuhne20}, who even gave an explicit form \cite[(42)]{BiluKuhne20} for the constant $c(1, n, d, h)$. The dependence of the constant $c(m, n, d, h)$ on $h$ means that \cite[Corollary~4]{Binyamini19} and \cite[Lemma~3.1]{BiluKuhne20} do not, for example, give an effective bound on the $n$-tuples $(x_1, \ldots, x_n)$ of pairwise distinct singular moduli $x_1, \ldots, x_n$ which satisfy \eqref{eq:lin} with $a_1, \ldots, a_n, b \in \Q$. 

It is also worth noting that it is possible to obtain (ineffective) bounds on the number of maximal special subvarieties which are uniform across definable families of subvarieties and do not depend on (the degree of) the field of definition of the subvariety. Scanlon's results on automatic uniformity \cite[Theorem~4.2]{Scanlon04} imply that, for every $m, n \in \Z_{>0}$, there exists an ineffective constant $c(m, n) > 0$ with the following property: if $a_1, \ldots, a_n \in \C \setminus \{0\}$ and $b \in \C$, then there are at most $c(m, n)$ distinct $n$-tuples $(x_1, \ldots, x_n)$ of pairwise distinct singular moduli $x_1, \ldots, x_n$ such that
\[ a_1 x_1^m + \ldots + a_n x_n^m = b.\]
In the case where $m=1$ and $n=2$, Bilu, Luca, and Masser \cite[Theorem~1.1]{BiluLucaMasser17} proved that there are only (ineffectively) finitely many distinct, non-special linear subvarieties of $\C^2$ which contain at least $3$ distinct special points.

\subsection{Structure of this paper}
In Section~\ref{sec:prelim}, we recall some facts needed throughout the paper. In Section~\ref{sec:ineff}, we explain how Theorem~\ref{thm:ineffpowers} follows from the results of Pila \cite{Pila11} and Binyamini \cite{Binyamini19}. Theorems~\ref{thm:equal} and~\ref{thm:field} are proved in Sections~\ref{sec:same} and~\ref{sec:fields} respectively. The proofs of Theorems~\ref{thm:effpowers} and \ref{thm:unifeffpowers} are then carried out in Section~\ref{sec:pf}. Section~\ref{sec:singmod} contains some properties of singular moduli, which are then used for the proof of Theorem~\ref{thm:trip} in Section~\ref{sec:trip}.

\subsection{Acknowledgements} I would like to thank Sebastian Eterovi\'{c} and Ziyang Gao for helpful comments. The author has received funding from the European Research Council (ERC) under the European Union’s Horizon 2020 research and innovation programme (grant agreement no. 945714).

\section{Preliminaries}\label{sec:prelim}

\subsection{Properties of singular moduli}\label{subsec:singmods}

We collect here some well-known properties of singular moduli which we will need throughout the paper. 

Let $j \colon \h \to \C$ denote the modular $j$-function, where $\h$ is the complex upper half plane. A \textbf{singular modulus} is a complex number $j(\tau)$, where $\tau \in \h$ is such that $[\Q(\tau) : \Q] = 2$. The \textbf{discriminant} $\Delta$ of a singular modulus $j(\tau)$ is given by $\Delta = b^2-4ac$, where $a, b, c \in \Z$, not all zero, are such that $a \tau^2 + b \tau + c = 0$ and $\gcd(a, b, c) = 1$. In particular, $\Delta < 0$ and $\Delta \equiv 0, 1 \bmod 4$. Hence, $\lvert \Delta \rvert \geq 3$ always. 

Note that $\Q(\tau) = \Q(\sqrt{\Delta})$. The \textbf{fundamental discriminant} $D$ of $j(\tau)$ is defined to be the discriminant of the imaginary quadratic field $\Q(\tau) = \Q(\sqrt{\Delta})$. One has that $\Delta = f^2 D$ for some $f \in \Z_{>0}$.

Write $F_j$ for the standard fundamental domain for the action (by fractional linear transformations) of $\SL$ on $\h$, i.e.
\begin{align*}
	F_j = &\{ z \in \h : -\frac{1}{2} \leq \re z < \frac{1}{2} \mbox{ and } \lvert z \rvert > 1\}\\
	&\cup \{ z \in \h : -\frac{1}{2} \leq \re z \leq 0 \mbox{ and } \lvert z \rvert = 1\}.
\end{align*}
The $j$-function restricts to a bijection $F_j \to \C$. Therefore, for $\Delta < 0$ such that $\Delta \equiv 0, 1 \bmod 4$, the map
\[ (a, b, c) \mapsto j\left( \frac{-b + \lvert \Delta \rvert^{1/2} i}{2a}\right)\]
is a bijection between the set
\begin{align*} 
	T_\Delta = \{&(a, b, c) \in \Z^3 : b^2 - 4 ac = \Delta, \, \gcd(a,b,c) = 1,\\ &\mbox{ and either } -a<  b \leq a <c \mbox{ or } 0 \leq b \leq a =c\}
\end{align*}
and the set of singular moduli of discriminant $\Delta$. For each such $\Delta$, there exists a unique triple $(a, b, c) \in T_\Delta$ with $a=1$, given by
\[ (a, b, c) = \left(1, k, \frac{k^2-\Delta}{4}\right),\]
where $k \in \{0, 1\}$ is such that $k \equiv \Delta \bmod 2$; we call the corresponding singular modulus the \textbf{dominant singular modulus} of discriminant $\Delta$.

 If $x$ is a singular modulus of discriminant $\Delta$, then, see \cite[Lemma 9.3 \& Theorem 11.1]{Cox22}, the field $\Q(\sqrt{\Delta}, x)$ is a Galois extension of both $\Q(\sqrt{\Delta})$ and $\Q$ and
 \[ \Gal\left(\Q\left(\sqrt{\Delta}, x\right) / \Q\left(\sqrt{\Delta}\right)\right) \cong \cl (\Delta),\]
 where $\cl(\Delta)$ denotes the class group of the unique imaginary quadratic order of discriminant $\Delta$.
 
 Let $x_1, \ldots, x_n$ be all the distinct singular moduli of some discriminant $\Delta$. Then, by \cite[Theorem 11.1 \& Proposition 13.2]{Cox22}, the polynomial
 \[ H_\Delta(z) = \prod_{i=1}^n (z- x_i)\]
 has coefficients in $\Z$ and is irreducible over $\Q$ and over $\Q(\sqrt{\Delta})$.
 Hence, for every $i \in \{1, \ldots, n\}$, we have that
\[ [\Q(x_i) : \Q] = \left[\Q\left(\sqrt{\Delta}, x_i\right) : \Q\left(\sqrt{\Delta}\right)\right] = h(\Delta),\]
where $h(\Delta)$ denotes the class number of the imaginary quadratic order of discriminant $\Delta$.

The class numbers of small discriminants may be computed straightforwardly in PARI. We will make use of the following consequence of this.

\begin{lem}\label{lem:smallclass}
	Let $x$ be a singular modulus of discriminant $\Delta$. If $\lvert \Delta \rvert < 15$, then $x \in \Q$. If $\lvert \Delta \rvert < 39$, then $h(\Delta)  \leq 3$.
\end{lem}

We will also need the following bound on singular moduli.

\begin{prop}\label{prop:bd}
	Let $x$ be a singular modulus of discriminant $\Delta$ which corresponds to a triple $(a, b, c) \in T_\Delta$. Then
	\[ \exp\left(\frac{\pi \lvert \Delta \rvert^{1/2}}{a}\right) - 2079 \leq \lvert x \rvert \leq \exp\left(\frac{\pi \lvert \Delta \rvert^{1/2}}{a}\right) + 2079.\]
\end{prop}

\begin{proof}
Let $(a, b, c) \in T_\Delta$ be the triple  corresponding to $x$. Then $x = j(\tau)$, where
	\[ \tau = \frac{-b + \lvert \Delta \rvert^{1/2} i}{2 a} \in F_j. \]
	The result follows immediately, since Bilu, Masser, and Zannier \cite[Lemma~1]{BiluMasserZannier13} proved that if $z \in F_j$, then
	\[ \left\lvert \left\lvert j(z) \right\rvert - \exp(2 \pi \im z) \right\rvert \leq 2079.\qedhere\] 
\end{proof}

This bound has the following consequence.

\begin{lem}\label{lem:dombig}
	Let $x, y$ be distinct singular moduli of the same discriminant $\Delta$. Suppose that $x$ is dominant. Then
	\[\lvert y \rvert \leq \frac{6 \lvert x \rvert}{\exp\left(\frac{\pi \lvert \Delta \rvert^{1/2}}{2}\right)}.\]
\end{lem}

\begin{proof}
	There are at least two distinct singular moduli of discriminant $\Delta$, so $\lvert \Delta \rvert \geq 15$ by Lemma~\ref{lem:smallclass}. Since $x$ is dominant, $y$ is not dominant.  Thus, Proposition~\ref{prop:bd} implies that
	\begin{align*}
		\frac{\lvert y \rvert}{\lvert x \rvert} &\leq  \frac{1}{\exp\left(\frac{\pi \lvert \Delta \rvert^{1/2}}{2}\right)} \frac{1 + 2079 \exp\left(\frac{-\pi \lvert \Delta \rvert^{1/2}}{2}\right)}{1 - 2079 \exp\left(- \pi \lvert \Delta \rvert^{1/2}\right)}\\ 
		&\leq \frac{1}{\exp\left(\frac{\pi \lvert \Delta \rvert^{1/2}}{2}\right)} \frac{1 + 2079 \exp\left(\frac{-\pi \sqrt{15}}{2}\right)}{1 - 2079 \exp\left(- \pi \sqrt{15}\right)}\\ 
		&\leq \frac{6}{\exp\left(\frac{\pi \lvert \Delta \rvert^{1/2}}{2}\right)}.\qedhere
		\end{align*}
\end{proof}

\subsection{Effective bounds for the class number}\label{subsec:Tatuzawa}

 A theorem of Tatuzawa \cite[Theorem 1]{Tatuzawa51} and Dirichlet's class number formula together imply the following result. It shows that Siegel's \cite{Siegel35} ineffective lower bound for the class number of imaginary quadratic fields may be made effective, apart from at most one possible exceptional imaginary quadratic field.

\begin{thm}[{\cite[Theorem 1]{Tatuzawa51}}]\label{thm:tatfund}
	Let $\epsilon \in (0, 1/2)$. Let $K_1, K_2$ be imaginary quadratic fields with respective discriminants $D_1, D_2$. If 
	\[ h(D_1) < \frac{\epsilon}{10 \pi} \lvert D_1 \rvert^{\frac{1}{2} - \epsilon} \mbox{ and } h(D_2) < \frac{\epsilon}{10 \pi} \lvert D_2 \rvert^{\frac{1}{2} - \epsilon},\]
	then $K_1 = K_2$.
\end{thm}

In other words, for a fixed $\epsilon \in (0, 1/2)$, there is at most one imaginary quadratic field $K$ for which the bound
\begin{align}\label{eq:Tat}
	 h(D) \geq \frac{\epsilon}{10 \pi} \lvert D \rvert^{\frac{1}{2} - \epsilon}
	 \end{align}
is false, where $D$ denotes the discriminant of $K$.
It is possible (indeed, it would follow from GRH) that, for some $\epsilon \in (0, 1/2)$, the bound \eqref{eq:Tat} in fact holds for every imaginary quadratic field $K$. In such a case, the results of \cite{Binyamini19} would imply an effective version of Theorem~\ref{thm:ineffpowers}.

Throughout this paper, we fix $\epsilon_* = 1/12$ and denote by $K_*$ the unique imaginary quadratic field for which the corresponding bound \eqref{eq:Tat} is false. Denote by $D_*$ the discriminant of $K_*$. If no such field $K_*$ exists, then we adopt the notational convention that the inequalities $K \neq K_*$ and $D \neq D_*$ hold for every imaginary quadratic field $K$ with discriminant $D$. We have the following consequence of Tatuzawa's result.

\begin{prop}[{\cite[(17)]{BiluKuhne20}}]\label{prop:Tatbd}
	Let $\Delta < 0$ be such that $\Delta \equiv 0, 1 \bmod 4$ and let $K = \Q(\sqrt{\Delta})$. If $K \neq K_*$, then
	\[ h(\Delta) \geq \frac{37}{50000} \lvert \Delta \rvert^{5/12}.\]
\end{prop}

For discriminants $\Delta$ with $\Q(\sqrt{\Delta}) = K_*$, the known effective bounds for the class number are much weaker.

\begin{prop}\label{prop:fundbd}
	Let $D < 0$ be a fundamental discriminant. Then
	\[ h(D) \geq \frac{1}{\sqrt{42000}} (\log \lvert D \rvert)^{1/2}.\]
\end{prop}

\begin{proof}
	Oesterl\'{e} \cite[Th\'{e}or\`{e}me~1 \& \S5.1]{Oesterle85}, building on work of Goldfeld \cite{Goldfeld76} and Gross--Zagier \cite[Theorem~8.1]{GrossZagier86}, proved  that
	\[ h(D) \geq \frac{\log \lvert D \rvert}{7000} F(D),\]
	where
	\[ F(D) = \prod \left( 1 - \frac{\lfloor 2 \sqrt{p} \rfloor}{p+1}\right)\]
	and this product is taken over all the primes $p$ which divide $D$ except the largest. Let $\omega(D)$ denote the number of distinct prime divisors of $D$. So there are $\omega(D) - 1$ terms in the product defining $F(D)$. By the theory of genera, see e.g. \cite[Proposition~3.11 \& Theorem~6.1]{Cox22}, we have that
	\[ \lvert \cl(D)[2] \rvert = 2^{\omega(D) -1}.\]
	In particular, $2^{\omega(D) -1} \mid h(D)$. So $\omega(D) - 1 \leq v_2(h(D))$, where $ v_2(h(D))$ denotes the largest integer $k$ such that $2^k \mid h(D)$. Hence, there are at most $v_2(h(D))$ terms appearing in the product defining $F(D)$. 
	
	Observe that
	\[ 1 - \frac{\lfloor 2 \sqrt{p} \rfloor}{p+1} \geq \frac{1}{2}\]
	if $p \geq 11$. Hence,
	\[F(D) \geq \frac{1}{3} \frac{1}{4} \frac{1}{3} \frac{3}{8} \left(\frac{1}{2}\right)^{v_2(h(D)) - 4} = \frac{1}{6} 2^{-v_2(h(D))}.\]
	Since $h(D) \geq 2^{v_2(h(D))}$ by definition, we obtain that
	\[ h(D)^2 \geq \frac{\log \lvert D \rvert}{42000}. \qedhere\]
	
\end{proof}

\begin{prop}\label{prop:class}
	Let $\Delta < 0$ be such that $\Delta \equiv 0, 1 \bmod 4$ and let $k \in \Z_{> 0}$. If $h(\Delta) \leq k$, then 
	\[  \lvert \Delta \rvert^{1/2}  \leq 2k^2 e^{21000k^2}.\]
\end{prop}

\begin{proof}
	Write $\Delta = f^2 D$, where $D$ is the discriminant of $\Q(\sqrt{\Delta})$. Suppose that $h(\Delta) \leq k$.
	Assume first that $D \notin \{-3, -4\}$. The class number formula in \cite[Theorem~7.24]{Cox22} gives that
	\[ h(\Delta) = h(D) f \prod_{p \mid f} \left(1 - \left(\frac{D}{p}\right)\frac{1}{p}\right),\]
	where $(D/p)$ denotes the Kronecker symbol. Hence, $h(D) \mid h(\Delta)$ and
	\[ h(\Delta) \geq h(D) \varphi(f),\]
	where $\varphi(\cdot)$ denotes Euler's totient function. Hence,
	\[h(D) \leq k \mbox{ and } \varphi(f) \leq k.\]
	Proposition~\ref{prop:fundbd} then implies that
	\[ \lvert D \rvert \leq e^{42000k^2}, \]
	while the classical bound $\varphi(f) \geq \sqrt{f/2}$ implies that $f \leq 2k^2$. Hence,
	\[ \lvert \Delta \rvert^{1/2} \leq 2k^2 e^{21000k^2}.\]
	
	Now assume $D \in \{-3, -4\}$. Then $h(D) = 1$. In this case, the class number formula from \cite[Theorem~7.24]{Cox22} implies that
	\[ h(\Delta) \geq  \frac{\varphi(f)}{3}.\]
	Hence, $\varphi(f) \leq 3 k$ and so $f \leq 18 k^2$. Thus, clearly,
	\[ \lvert \Delta \rvert^{1/2} \leq 36k^2 \leq 2 k^2 e^{21000k^2}. \qedhere\]
\end{proof}

Finally, we will also need an upper bound for the class number. This bound is a straightforward consequence of Dirichlet's class number formula.

\begin{prop}\label{prop:upperclass}
	If $\Delta < 0$ is such that $\Delta \equiv 0, 1 \bmod 4$, then $h(\Delta) \leq \lvert \Delta \rvert^{2/3}$.
\end{prop}

\begin{proof}
	By \cite[Proposition~2.2]{Paulin16}, we have that
	\[ h(\Delta) \leq \frac{1}{\pi} \lvert \Delta \rvert^{1/2} (2 + \log \lvert \Delta \rvert).\]
	It then suffices to note that, 	for every $x>0$,
	\[ \frac{2 + \log x}{\pi} < x^{1/6}. \qedhere\]
\end{proof}

\section{Deducing Theorem~\ref{thm:ineffpowers}}\label{sec:ineff}

In this section, we explain how Theorem~\ref{thm:ineffpowers} may be deduced from the results of Pila \cite[Theorem~13.2]{Pila11} and Binyamini \cite[Corollary~4]{Binyamini19}. A direct proof of the $m = 1$ case of Theorem~\ref{thm:ineffpowers} was given by Pila \cite[Theorem~7.7]{Pila14a}.

For the definition of a special subvariety of $\C^n$, see e.g. \cite[Definition~1.3]{Pila11}, \cite[(1.1)]{Binyamini19}, or \cite[Definition~4.10]{Pila22}. For a subvariety $V \subset \C^n$, denote by $V^\mathrm{sp}$ the union of all the positive-dimensional special subvarieties of $V$. 

Let $m, n \in \Z_{>0}$. For $a = (a_1, \ldots, a_n) \in \C^n$ and $b \in \C$, let
\[ V_{a, b} = \left\{\left(z_1, \ldots, z_n\right) \in \C^n :a_1 z_1^m + \ldots + a_n z_n^m = b \right\}.\]

\begin{lem}\label{lem:specials}
	Let $m, n \in \Z_{>0}$. Let $x_1, \ldots, x_n$ be pairwise distinct singular moduli. Suppose that $a = (a_1, \ldots, a_n) \in (\alg \setminus \{0\})^n$ and $b \in \alg$ are such that
	\[ a_1 x_1^m + \ldots + a_n x_n^m = b.\]
	Then
	\[ (x_1, \ldots, x_n) \in V_{a, b} \setminus V_{a, b}^\mathrm{sp}.\]
\end{lem}

\begin{proof}
	Clearly, 
	\[ (x_1, \ldots, x_n) \in V_{a, b}. \]
	Suppose that 
	\[(x_1, \ldots, x_n) \in S\]
	for some positive-dimensional special subvariety $S$ of $V_{a, b}$.	We show that this is impossible. 
	
	The subvariety $V_{a, b}$ is, in the terminology of \cite[Definition~2]{Binyamini19}, a hereditarily degree non-degenerate (hdnd) hypersurface. Hence, by \cite[Corollary~4]{Binyamini19}, the special subvariety $S$ may be defined by equations solely of the form $z_i = z_k$ and $z_l = x_l$. Since $x_1, \ldots, x_n$ are pairwise distinct, no non-trivial equations $z_i = z_k$ hold on $S$. Thus, up to reordering the coordinates,
	\[ S = \prod_{i \in I} \{x_i\} \times \C^{n-k}\]
	for some $k \in \{1, \ldots, n-1\}$ and $I \subset \{1, \ldots, n\}$ with $\lvert I \rvert = k$. So
	\[ \sum_{i \in I} a_i x_i^m + \sum_{i \in \{1, \ldots, n\} \setminus I} a_i z_i^m = b\]
	for all $z_i \in \C$, which is clearly absurd since the $a_i$ are non-zero.
\end{proof}

\begin{proof}[Proof of Theorem~\ref{thm:ineffpowers}]
Let $m, n \in \Z_{>0}$. Define $V_{a, b} \subset \C^n$ as above. Let
\[ \mathbb{V} = \{ (z, a, b) \in \C^n \times \C^n \times \C : z \in V_{a, b}\}.\]
View $\mathbb{V}$ as a definable family of subvarieties of $\C^n$ with fibres $V_{a, b}$. Note that the definition of $\mathbb{V}$ depends only on $m, n$.

By Pila's Uniform Andr\'e--Oort for $\C^n$ \cite[Theorem~13.2]{Pila11} applied to $\mathbb{V}$, for every $d \in \Z_{>0}$, there exists an ineffective constant $c(m, n, d)>0$ with the following property. Let $x_1, \ldots, x_n$ be singular moduli and write $\Delta_i$ for the discriminant of $x_i$. Let $a = (a_1, \ldots, a_n) \in (\alg \setminus \{0\})^n$ and $b \in \alg$. If
\[ (x_1, \ldots, x_n) \in V_{a, b} \setminus V_{a, b}^\mathrm{sp}\]
and 
\[ [\Q(a_1, \ldots, a_n, b) : \Q] \leq d,\]
then
\[ \max \{ \lvert \Delta_i \rvert : i=1, \ldots, n\} \leq c(m, n, d).\]
Theorem~\ref{thm:ineffpowers} thus follows by Lemma~\ref{lem:specials}.
\end{proof}

\subsection{Uniformity in Theorem~\ref{thm:ineffpowers}}\label{subsec:dep}
	Recall, from Section~\ref{subsec:singmods}, that the singular moduli of a given discriminant $\Delta$ form a complete set of Galois conjugates over $\Q$. Moreover, by Proposition~\ref{prop:class}, the number of distinct singular moduli of discriminant $\Delta$ is $\geq c_1 (\log \lvert \Delta \rvert)^{c_2}$ for some absolute effective constants $c_1, c_2>0$. Therefore, if $m, n$ are fixed, then $c(m, n, d) \to \infty$ as $d \to \infty$, where $c(m, n, d)$ is the constant in Theorem~\ref{thm:ineffpowers}. Similarly, if $m, d$ are fixed, then $c(m, n, d) \to \infty$ as $n \to \infty$.
	
	On the other hand, if $n, d$ are fixed, then it is not obvious what happens to the constant $c(m, n, d)$ as $m \to \infty$. Since singular moduli are algebraic, one cannot bound the discriminants associated to a special point lying on a general hypersurface in $\C^n$ solely in terms of $n$ and the minimal degree of a field of definition of the hypersurface. However, it may be possible to obtain bounds that are uniform in $m$ for the specific family of hypersurfaces considered in Theorem~\ref{thm:ineffpowers}, i.e.~those defined by equations of the form
	\[ a_1 x_1^m + \ldots + a_n x_n^m = b.\]

	Indeed, the constant $c(m, n, d)$ in Theorem~\ref{thm:ineffpowers} may be taken to be uniform in $m$ if either $n = 1$ or $(d, n) = (1, 2)$. If $x$ is a singular modulus of discriminant $\Delta$ such that
	\[a x^m = b\]
	for some $a, b \in \alg \setminus \{0\}$ and $m \in \Z_{>0}$, then $\lvert \Delta \rvert$ may be bounded solely in terms of $[\Q(a, b) : \Q]$, thanks to Proposition~\ref{prop:class} and the fact \cite[Lemma~2.6]{Riffaut19} that $\Q(x^m) = \Q(x)$. If $x_1, x_2$ are distinct singular moduli of respective discriminants $\Delta_1, \Delta_2$ such that
	\[ a_1 x_1^{m_1} + a_2 x_2^{m_2} \in \Q\]
	for some $a_1, a_2 \in \Q \setminus \{0\}$ and $m_1, m_2 \in \Z_{>0}$, then $\max \{\lvert \Delta_1 \rvert, \lvert \Delta_2 \rvert \} \leq 427$
	by \cite[Theorem~1.3]{LucaRiffaut19} (see also \cite[Theorem~1.5]{Riffaut19}). For $n = 3$, we have the previous result of the author \cite[Theorem~1.1]{Fowler23}: if $x_1, x_2, x_3$ are pairwise distinct singular moduli of respective discriminants $\Delta_1, \Delta_2, \Delta_3$ such that
	\[ a_1 x_1^m + a_2 x_2^m + a_3 x_3^m \in \Q\]
	for some $m \in \Z_{>0}$ and $a_1, a_2, a_3 \in \Q \setminus \{0\}$ with $\lvert a_1 \rvert = \lvert a_2 \rvert = \lvert a_3 \rvert$, then $\max \{\lvert \Delta_1 \rvert, \lvert \Delta_2 \rvert, \lvert \Delta_3 \rvert \} \leq 907$.

\section{Equations in singular moduli of the same discriminant}\label{sec:same}

In this section, we prove Theorem~\ref{thm:equal}.

	\begin{proof}[Proof of Theorem~\ref{thm:equal}]
		Let $n \in \Z_{>0}$. Let $x_1, \ldots, x_n$ be pairwise distinct singular moduli of discriminant $\Delta$. In particular, $h(\Delta) \geq n$. Write $K = \Q(\sqrt{\Delta})$. Suppose that
		\[ a_1 x_1^m + \ldots + a_n x_n^m = b\]
		for some $a_1, \ldots, a_n \in K \setminus \{0\}$, $b \in K$, and $m \in \Z_{>0}$. 
			
			Suppose first that $h(\Delta) \geq n+1$. The singular moduli of discriminant $\Delta$ form a complete set of Galois conjugates over $K$. Hence, there must exist $K$-conjugates
			\[(x_{1, i}, \ldots, x_{n, i})\]
			of $(x_1, \ldots, x_n)$, where $i \in \{1, \ldots, n+1\}$, with the property that:
	\[x_{k, i} \mbox{ is dominant } \iff k=i.\]
	Note that
	\[a_1 x_{1, i}^m + \ldots + a_n x_{n, i}^m = b\]
	for every $i \in \{1, \ldots, n+1\}$. Therefore, 
	\begin{align*}
		\begin{vmatrix}
			1 & \cdots & 1\\
			x_{1, 1}^m & \cdots & x_{1, n+1}^m\\
			\vdots & & \vdots\\
			x_{n, 1}^m & \cdots & x_{n, n+1}^m
		\end{vmatrix}
	=0.
		\end{align*}
	Expanding this determinant, we have that
	\[ \sum_{\sigma \in S_{n+1}} \mathrm{sgn}(\sigma) (x_{1, \sigma(2)} \cdots x_{n, \sigma(n+1)})^m = 0.\]
	Let $\tau$ be the unique element of $S_{n+1}$ such that $\tau(i+1) = i$ for every $i \in \{1, \ldots, n\}$. Then
	\begin{align}\label{eq:equal1}
		\lvert x_{1, 1} \cdots x_{n, n} \rvert^m = \left \lvert \sum_{\sigma \in S_{n+1} \setminus \{ \tau\}} \mathrm{sgn}(\sigma) (x_{1, \sigma(2)} \cdots x_{n, \sigma(n+1)})^m \right\rvert.
		\end{align}
	
	Note that
	\begin{align*}
		\lvert x_{1, 1} \cdots x_{n, n} \rvert = \lvert x_{1, 1} \rvert^n,
		\end{align*}
	since $x_{i, i}$ is the unique dominant singular modulus of discriminant $\Delta$ for every $i \in \{1, \ldots, n\}$. Observe also that if $\sigma \in S_{n+1} \setminus \{ \tau\}$, then $\sigma(i+1) \neq i$ for some $i \in \{1, \ldots, n\}$ and hence at least one of
	\[ x_{1, \sigma(2)}, \ldots, x_{n, \sigma(n+1)}\]
		is not dominant. Therefore, by Lemma~\ref{lem:dombig},
		\begin{align}\label{eq:equal2}
		&\left\lvert 	\sum_{\sigma \in S_{n+1} \setminus \{ \tau\}} \mathrm{sgn}(\sigma) (x_{1, \sigma(2)} \cdots x_{n, \sigma(n+1)})^m \right\rvert \nonumber\\ 
		\leq &(n+1)! \left(\frac{6   }{\exp\left(\frac{\pi \lvert \Delta \rvert^{1/2}}{2}\right) } \right)^m\lvert x_{1, 1} \rvert^{mn}.
			\end{align}
		Then \eqref{eq:equal1} and \eqref{eq:equal2} together imply that
		\[\left(\frac{\exp\left(\frac{\pi \lvert \Delta \rvert^{1/2}}{2}\right)}{6}\right)^m \leq  (n+1)!.\]
		Since $\lvert \Delta \rvert \geq 3$ for every discriminant $\Delta$, we have that
		\[ \exp\left(\frac{\pi \lvert \Delta \rvert^{1/2}}{2}\right) \geq \exp\left(\frac{\pi \sqrt{3}}{2}\right) = 15.190\ldots > 6.\]
		Thus,
			\[\frac{\exp\left(\frac{\pi \lvert \Delta \rvert^{1/2}}{2}\right)}{6} \leq  (n+1)!.\]
		For every $k \in \Z_{>0}$, Stirling's formula \cite[(1), (2)]{Robbins55} implies that
		\[ k! \leq \sqrt{2 \pi} \exp\left(\left(k+ \frac{1}{2}\right) \log\left(k\right) - k + \frac{1}{12}\right).\]
		Hence, 
		\begin{align*}
			 \lvert \Delta \rvert^{1/2} &\leq \frac{1}{\pi} \left((2n+3) \log(n+1) -2(n+1) +\log(72 \pi) + \frac{1}{6}\right)\\
			 &\leq \frac{1}{\pi} \left(\left(2n+3\right) \log\left(n+1\right) -2n +4\right).
			 \end{align*}
		
		Now suppose that $h(\Delta) = n$. So $x_1, \ldots, x_n$ are a complete set of Galois conjugates over $K$. Newton's identities thus imply that
		\[ x_1^m + \ldots + x_n^m \in K.\]
		Hence,
		\[ (a_2 - a_1) x_2^m + \ldots + (a_n - a_1) x_n^m \in K. \]
	If $a_1, \ldots, a_n$ are not all equal, then $a_i \neq a_1$ for some $i \in \{2, \ldots, n\}$. Let
	\[ l = \# \{i \in \{2, \ldots, n\} : a_i \neq a_1\}.\]
		Since $l < n$, we have that $h(\Delta) \geq l+1$. We may then apply, for $l$ singular moduli, the already proved case of the theorem to obtain that
		\begin{align*} 
			\lvert \Delta \rvert^{1/2} &\leq \frac{1}{\pi} \left((2l+3) \log(l+1) -2l +4\right)\\
			&\leq \frac{1}{\pi} \left((2n+3) \log(n+1) -2n +4\right).\qedhere
			\end{align*}
		\end{proof}
	
	\begin{remark}\label{rmk:q}
		In the proof of Theorem~\ref{thm:equal}, it is crucial that every singular modulus of discriminant $\Delta$ is conjugate over $K$ to the dominant singular modulus of discriminant $\Delta$. Thus, the above proof does not extend to equations over an arbitrary number field $L$. In contrast, the deduction, in the next two sections, of Theorem~\ref{thm:field} and then Theorem~\ref{thm:effpowers} from Theorem~\ref{thm:equal} would go through over a general number field $L$.
	\end{remark}
	
	\section{Fields generated by singular moduli}\label{sec:fields}
	
Theorem~\ref{thm:field} may be deduced from Theorem~\ref{thm:equal}, as we now show.

\begin{proof}[Proof of Theorem~\ref{thm:field}]
	Let $n \in \Z_{>0}$. Let $x_1, \ldots, x_n$ be pairwise distinct singular moduli of the same discriminant $\Delta$. Let $K = \Q(\sqrt{\Delta})$.	Suppose that $a_1, \ldots, a_n \in K \setminus \{0\}$  and $m \in \Z_{>0}$ are such that
	\[\left[K\left(x_1, a_1 x_1^m + \ldots + a_n x_n^m\right) : K\left(a_1 x_1^m + \ldots + a_n x_n^m\right) \right] > n.\]
	We will show that the desired bound holds on $\lvert \Delta \rvert$.
	
	Recall that $K(x_1) / K$ is a Galois extension. In particular, 
	\[  K(x_1, a_1 x_1^m + \ldots + a_n x_n^m) = K(x_1),\] 
	since $x_1, \ldots, x_n$ are all conjugate over $K$. Let $G = \Gal(K(x_1) / K)$ and let 
\[H = \Gal\left(K\left(x_1\right)/ K\left(a_1 x_1^m + \ldots + a_n x_n^m\right)\right).\] So
	\begin{align*}
		 \lvert H \rvert &= \left[K\left(x_1\right) : K\left(a_1 x_1^m + \ldots + a_n x_n^m\right) \right] > n
		 \end{align*}
	 by assumption. In particular, there exists some $\sigma \in H$ such that 
	 \[\sigma(x_1) \notin \{x_1, \ldots, x_n\}.\]
	 Since $\sigma$ fixes $K(a_1 x_1^m + \ldots + a_n x_n^m)$, we have that
	 \[ a_1 x_1^m + \ldots + a_n x_n^m = a_1 \sigma(x_1)^m + \ldots + a_n \sigma(x_n)^m.\]
	 The $x_1, \ldots, x_n$ are pairwise distinct and the $\sigma(x_1), \ldots, \sigma(x_n)$ are pairwise distinct too. Cancelling terms as necessary, there must exist some $k \in \{2, \ldots, 2n\}$ and pairwise distinct singular moduli $y_1, \ldots, y_k$ of discriminant $\Delta$ such that
	 \[ b_1 y_1^m + \ldots + b_k y_k^m = 0\]
	 for some $b_1, \ldots, b_k \in K \setminus \{0\}$. Here we used that $\sigma(x_1) \notin \{x_1, \ldots, x_n\}$. Hence, by Theorem~\ref{thm:equal}, either: $\lvert \Delta \rvert^{1/2} \leq c_1(k)$ and so certainly $\lvert \Delta \rvert^{1/2} \leq c_1(2n)$ as desired (where $c_1(\cdot)$ denotes the constant from Theorem~\ref{thm:equal}), or: $h(\Delta) = k$ and $b_1 = \ldots = b_k$. 
	 
	 Suppose then that $h(\Delta) = k$ and $b_1 = \ldots = b_k$. Re-indexing the $y_i$ as necessary, we may assume that $y_1$ is dominant. Note also that $y_1 \neq 0$ and $\lvert \Delta \rvert \geq 15$, by Lemma~\ref{lem:smallclass} since $k \geq 2$. Then, by Lemma~\ref{lem:dombig} and Proposition~\ref{prop:upperclass}, 
	 \begin{align*}
	 	&\lvert b_1 y_1^m + \ldots + b_k y_k^m \rvert\\ \geq &\left\lvert b_1 \right\rvert \left( \left\lvert y_1 \right\rvert^m - \left( \left\lvert y_2 \right\rvert^m + \ldots + \left\rvert y_k \right\rvert^m \right) \right) \\
	 	\geq &\left\lvert b_1 \right \rvert \left \lvert y_1 \right\rvert^m \left(1-   \frac{\lvert \Delta \rvert^{1/2} (2 + \log \lvert \Delta \rvert)}{\pi}\left(\frac{6}{ \exp\left(\frac{\pi \lvert \Delta \rvert^{1/2}}{2}\right)}\right)^m\right)\\
	 	\geq &\left\lvert b_1 \right \rvert \left \lvert y_1 \right\rvert^m \left(1-   \frac{6 \lvert \Delta \rvert^{1/2} (2 + \log \lvert \Delta \rvert) }{ \pi \exp\left(\frac{\pi \lvert \Delta \rvert^{1/2}}{2}\right)}\right)\\
	 	>&0,
	 	\end{align*}
 	which is a contradiction. So we must have that $\lvert \Delta \rvert^{1/2} \leq c_1(2n)$. 
	\end{proof}
	
	\section{The proof of Theorems~\ref{thm:effpowers} and \ref{thm:unifeffpowers}}\label{sec:pf}

\subsection{Completing the proof of Theorem~\ref{thm:effpowers}}

In this section, we complete the proof of Theorem~\ref{thm:effpowers}. We start by recalling Binyamini's \cite{Binyamini19} ``almost'' effective Andr\'e--Oort result for $\C^n$. For the definition of a special subvariety of $\C^n$, see e.g. \cite[(1.1)]{Binyamini19} or \cite[Definition~4.10]{Pila22}. For a subvariety $V \subset \C^n$, we denote by $\deg V$ the degree of $V$ with respect to the projective embedding $\C^n \subset \mathbb{P}^n$. Binyamini proved the following theorem\footnote{To be precise, Binyamini \cite[\S2.2]{Binyamini19} fixed $\epsilon_* = 0.01$, but it is clear his result also holds with our choice that $\epsilon_* = 1/12$.}. (Recall that $K_*$ denotes the exceptional imaginary quadratic field that was defined above Proposition~\ref{prop:Tatbd}.)

\begin{thm}[{\cite[Theorem~1]{Binyamini19}}]\label{thm:effAO}
	Let $n, d, l \in \Z_{>0}$. There exists an effective constant $c(n, d, l)>0$ with the following property: 
	
	Let $V \subset \C^n$ be a subvariety defined over some number field $L$ with $[L : \Q] \leq d$ and $\deg V \leq l$. Denote by $V^\mathrm{sp}$ the union of all the positive-dimensional special subvarieties of $V$. If $x_1, \ldots, x_n$ are singular moduli with respective discriminants $\Delta_1, \ldots, \Delta_n$ such that
	\[(x_1, \ldots, x_n) \in V \setminus V^\mathrm{sp},\]
	then, for every $i \in \{1, \ldots, n\}$, either
	\[ \lvert \Delta_i \rvert \leq c\left(n, d, l\right)\]
	or
	\[ \Q\left(\sqrt{\Delta_i}\right) = K_*.\]
\end{thm}

Let $m, n \in \Z_{>0}$. For $a = (a_1, \ldots, a_n) \in (\Q \setminus \{0\})^n$ and $b \in \Q$, let
\[ V_{a, b} = \left\{\left(z_1, \ldots, z_n\right) \in \C^n :a_1 z_1^m + \ldots + a_n z_n^m = b \right\}.\]
Since the $V_{a, b}$ are all defined over $\Q$ and have degree $m$, the constant in Theorem~\ref{thm:effAO} is uniform (for the given $m, n$) across the $V_{a, b}$.

\begin{proof}[Proof of Theorem~\ref{thm:effpowers}]
	Let $m, n \in \Z_{>0}$. Let $c_1(m, n)>0$ be the effective constant given by Theorem~\ref{thm:effAO} applied to the family of linear subvarieties $V_{a, b} \subset \C^n$, where $a \in (\Q \setminus \{0\})^n$ and $b \in \Q$ (i.e. $c_1(m, n) = c(n, 1, m)$ for the constant $c(n, d, l)$ in the statement of Theorem~\ref{thm:effAO}).
	
	Suppose that $x_1, \ldots, x_n$ are pairwise distinct singular moduli such that
	\[ a_1 x_1^m + \ldots + a_n x_n^m = b\]
	for some $a_1, \ldots, a_n \in \Q \setminus \{0\}$ and $b \in \Q$. Then, by Lemma~\ref{lem:specials},
	\[(x_1, \ldots, x_n) \in V_{a, b} \setminus V_{a, b}^\mathrm{sp}.\]
	Hence, by Theorem~\ref{thm:effAO}, for every $i \in \{1, \ldots, n\}$, either
	\[ \lvert \Delta_i \rvert \leq c_1(m, n)\]
	or
	\[ \Q\left(\sqrt{\Delta_i}\right) = K_*.\]
	
	If there is no $i$ such that $\Q(\sqrt{\Delta_i}) = K_*$, then we are done. We may thus assume that there exists a discriminant $\Delta_*$ such that $\Q(\sqrt{\Delta_*}) = K_*$ and, for every $i \in \{1, \ldots, n\}$, either $\Delta_i = \Delta_*$ or $\Q(\sqrt{\Delta_i}) \neq K_*$.
	Relabelling as necessary, we may assume that 
	there exists some $k \in \{1, \ldots, n\}$ such that
	\[ \Delta_i = \Delta_* \iff i \in \{1, \ldots, k\}.\]
	In particular, if $k+1 \leq i \leq n$, then $\Q(\sqrt{\Delta_i}) \neq K_*$ and so
	\[ \lvert \Delta_i \rvert \leq c_1(m, n).\]
	
	Observe that
	\[ \sum_{i=1}^k a_i x_i^m = b - \sum_{i = k+1}^n a_i x_i^m,\]
	where the sum on the right hand side may be empty. Hence,
	\[ \Q\left(a_1 x_1^m + \ldots + a_k x_k^m \right) \subset \Q\left(\left\{x_i : k+1 \leq i \leq n\right\}\right).\]
	Since the $\lvert \Delta_i \rvert$ are effectively bounded (solely in terms of $m, n$) for $i \geq k+1$, there exists, by Proposition~\ref{prop:upperclass}, an effective constant $c_2(m, n)>0$ such that
	\[ \left[ \Q\left(\left\{x_i : k+1 \leq i \leq n\right\}\right): \Q\right] \leq c_2(m, n).\]
	Hence,
	\[ \left [\Q\left(a_1 x_1^m + \ldots + a_k x_k^m \right) : \Q \right] \leq c_2(m, n).\]
	
	Since $\Delta_1 = \ldots = \Delta_k = \Delta_*$, we may apply Theorem~\ref{thm:field} to see that there exists an effective constant $c_3(n)>0$ such that either
	\[ \lvert \Delta_* \rvert \leq c_3(n)\]
	or
	\[ [\Q(x_1, a_1 x_1^m + \ldots + a_k x_k^m) : \Q(a_1 x_1^m + \ldots + a_k x_k^m)] \leq 2k \leq 2n.\]
	In the first case, we are done. So assume then that
	\[ [\Q(x_1, a_1 x_1^m + \ldots + a_k x_k^m) : \Q(a_1 x_1^m + \ldots + a_k x_k^m)] \leq 2n.\]
	Then
	\begin{align*} 
		h(\Delta_*) =	&[\Q(x_1) : \Q]\\
		\leq &[\Q(x_1, a_1 x_1^m + \ldots + a_k x_k^m) : \Q]\\
		= &[\Q(x_1, a_1 x_1^m + \ldots + a_k x_k^m) : \Q(a_1 x_1^m + \ldots + a_k x_k^m)]\\
		&[\Q(a_1 x_1^m + \ldots + a_k x_k^m) : \Q]\\
		\leq &2n c_2(m, n).
	\end{align*}
	Thus, by Proposition~\ref{prop:class}, there exists an effective constant $c_4(m, n) > 0$, which depends only on $m, n$, such that $\lvert \Delta_* \rvert \leq c_4(m, n)$. 
\end{proof}

\subsection{The proof of Theorem~\ref{thm:unifeffpowers}}\label{subsec:unipf}

The first part of Theorem~\ref{thm:unifeffpowers} (i.e.~the existence of the constant $c_1(n)$) is given by the following proposition. It follows from Theorem~\ref{thm:field} by an argument of Bilu and K\"uhne \cite[\S3, Step 4]{BiluKuhne20}, which in turn relies on an earlier result of K\"uhne \cite[Corollary~1.2]{Kuhne21}.

\begin{prop}\label{prop:step4}
	Let $x_1, \ldots, x_n$ be pairwise distinct singular moduli such that
	\[ a_1 x_1^m + \ldots + a_n x_n^m = b\]
	for some $m \in \Z_{>0}$ and $a_1, \ldots, a_n \in \Q \setminus \{0\}$ and $b \in \Q$.
	Write $\Delta_i$ for the discriminant of $x_i$ and $D_i$ for the fundamental discriminant. 
	Suppose that $k \in \{1, \ldots, n\}$ is such that $D_k \neq D_*$ and
	\[ \# \{\Delta_i : i \in \{1, \ldots, n\} \mbox{ such that } D_i = D_k\} = 1.\] 
	Then
	\[ \lvert \Delta_k \rvert^{1/2} \leq (3.8 \times 10^{10})(2.1 \times 10^4)^n (n+1)^{4n+6}.\]
\end{prop}

\begin{proof}
	Let $k \in \{1, \ldots, n\}$ be such that $D_k \neq D_*$ and
	\[ \# \{\Delta_i : i \in \{1, \ldots, n\} \mbox{ such that } D_i = D_k\} = 1.\] 
	Let
	\[ r = \# \{D_1, \ldots, D_n\}.\]
	
	Suppose first that $r=1$. Then $\Delta_1 = \ldots = \Delta_n$ by assumption. Hence, by Theorem~\ref{thm:equal}, either 
	\[ \lvert \Delta_k \rvert^{1/2} \leq \frac{1}{\pi} \left(\left(2n+3\right) \log\left(n+1\right) -2n +4\right),\]
	or $h(\Delta) = n$. In the latter case, since $D_k \neq D_*$, Proposition~\ref{prop:Tatbd} implies that
	\[ \lvert \Delta_k \rvert^{1/2} \leq \left(\frac{50000 n}{37}\right)^{6/5}.\]
	Thus, in either case, the desired bound on $\lvert \Delta_k \rvert^{1/2}$ evidently holds.
	
	So we may assume that $r \geq 2$. Relabelling as necessary, we may assume that $k = 1$ and there are $0 = n_0 < n_1 < \ldots < n_r = n$ such that
	\begin{align*}
		D_{n_0 + 1} = &\ldots = D_{n_1},\\
		D_{n_1 + 1} = &\ldots = D_{n_2},\\
		&\ldots\\
		D_{n_{r-1} + 1} = &\ldots = D_{n_r}.
	\end{align*}
	For $i \in \{1, \ldots, r\}$, let $K_i = \Q(\sqrt{D_{n_i}})$ and
	\[ L_i = K_i(a_{n_{i-1} + 1} x_{n_{i-1} + 1}^m + \ldots + a_{n_i} x_{n_i}^m).\]
	
	By assumption, $\Delta_{1} = \ldots = \Delta_{n_1}$.
	We may also assume that
	\[\lvert \Delta_1 \rvert^{1/2} > \frac{1}{\pi} \left(\left(4n+3\right) \log\left(2n+1\right) -4n + 4\right),\]
	since otherwise the desired bound already holds.
	Hence, by Theorem~\ref{thm:field},
	\[ [L_1(x_1) : L_1] \leq n_1.\]
	Since $D_1 \neq D_*$, one may now argue exactly as below \cite[(39)]{BiluKuhne20} to obtain that
	\[ \lvert \Delta_1 \rvert^{1/2} \leq (3.8 \times 10^{10})(2.1 \times 10^4)^n (n+1)^{4n+6}. \qedhere\]
\end{proof}

The second part of Theorem~\ref{thm:unifeffpowers} is an easy consequence of Theorem~\ref{thm:field} and Proposition~\ref{prop:step4}.

\begin{proof}[Proof of Theorem~\ref{thm:unifeffpowers}]
	Let $n \in \Z_{>0}$. Let $x_1, \ldots, x_n$ be pairwise distinct singular moduli. 
	Write $\Delta_i$ for the discriminant of $x_i$ and $D_i$ for the fundamental discriminant. 
	Suppose that
	\[ a_1 x_1^m + \ldots + a_n x_n^m = b\]
	for some $m \in \Z_{>0}$ and $a_1, \ldots, a_n \in \Q \setminus \{0\}$ and $b \in \Q$.
	Suppose that, for every	$k \in \{1, \ldots, n\}$,
	\[  \# \left \{ \Delta_i :  \, i \in \left \{1, \ldots, n\right \} \mbox{ such that } D_i = D_k \right \} = 1.\]
	
	Hence, by Proposition~\ref{prop:step4}, if $i \in \{1, \ldots, n\}$ is such that $D_i \neq D_*$, then
	\[ \lvert \Delta_i \rvert^{1/2} \leq (3.8 \times 10^{10})(2.1 \times 10^4)^n (n+1)^{4n+6}.\]
	So if $D_* \notin \{D_1, \ldots, D_n\}$, then we are done.

Without loss of generality, assume then that $D_1 = D_*$. For $i \geq 2$, either $\Delta_i = \Delta_1$, or $D_i \neq D_*$ and so $\lvert \Delta_i \rvert$ is bounded as above. Relabelling as necessary, there exists $k \in \{1, \ldots, n\}$ such that:
\[ \Delta_1 = \ldots = \Delta_{k}\]
and if $i > k$, then $D_i \neq D_*$ and hence
\[ \lvert \Delta_i \rvert^{1/2} \leq (3.8 \times 10^{10})(2.1 \times 10^4)^n (n+1)^{4n+6}.\]
Hence, by Proposition~\ref{prop:upperclass}, if $i > k$, then
\[[\Q(x_i) : \Q]  \leq \left((3.8 \times 10^{10})(2.1 \times 10^4)^n (n+1)^{4n+6}\right)^{4/3}.\]

Let
\[ L = \Q(a_1 x_1^m + \ldots + a_k x_k^m).\]
We may and do assume that
\[ \lvert \Delta_1 \rvert^{1/2} > \frac{1}{\pi} \left(\left(4k+3\right) \log\left(2k+1\right) -4k + 4\right).\]
Since $\Delta_1 = \ldots = \Delta_k$, Theorem~\ref{thm:field} implies that
\[ [L(\sqrt{\Delta_1}, x_1) : L(\sqrt{\Delta_1})] \leq k \leq n.\]
Thus,
\[[\Q(x_1) : \Q] \leq [L(\sqrt{\Delta_1}, x_1) : \Q] \leq 2n[L : \Q].\]

Since
\[ a_1 x_1^m + \ldots + a_k x_k^m = b - \sum_{i = k + 1}^n a_{i} x_{i}^m ,\]
we have that
\[ L \subset \Q(\{x_i : k < i \leq n\}).\]
Hence,
\[ [L : \Q] \leq \prod_{i=k+1}^n [\Q(x_i) : \Q] \leq \left( \left((3.8 \times 10^{10})(2.1 \times 10^4)^n (n+1)^{4n+6}\right)^{4/3} \right)^n,\]
and so
\[ [\Q(x_1) : \Q] \leq 2n\left((3.8 \times 10^{10})(2.1 \times 10^4)^n (n+1)^{4n+6}\right)^{4n/3}.\]
Hence, if we write
\[ g(n) = \left(2n\left((3.8 \times 10^{10})(2.1 \times 10^4)^n (n+1)^{4n+6}\right)^{4n/3}\right)^2\]
for convenience, then Proposition~\ref{prop:class} implies that
\[ \lvert \Delta_1 \rvert^{1/2} \leq 2 g(n) e^{21000 g(n)}. \qedhere\]
\end{proof}

\section{Some more properties of singular moduli}\label{sec:singmod}

In this section, we collect some more facts about singular moduli, which we will use for the proof of Theorem~\ref{thm:trip} in Section~\ref{sec:trip}.

Recall, from Section~\ref{subsec:singmods}, that a singular modulus $x$ of discriminant $\Delta$ corresponds to a unique triple $(a, b, c) \in T_\Delta$. Call the number $a$ the \textbf{denominator} of the corresponding singular modulus. A singular modulus is called \textbf{dominant} (respectively \textbf{subdominant}) if it has denominator $1$ (respectively $2$). This terminology is adopted from \cite{BiluGunTron22, FayeRiffaut18, Riffaut19}. We need the following result on the number of singular moduli of a given discriminant with small denominators.

\begin{prop}[{\cite[Proposition~2.6]{BiluLucaMadariaga16} \& \cite[Lemma~2.1]{Fowler23}}]\label{prop:denom}
	Let $\Delta < 0$ be such that $\Delta \equiv 0, 1 \bmod 4$. Then:
	\begin{enumerate}
		\item There exists exactly one dominant singular modulus of discriminant $\Delta$.
		\item For each $a \in \{2, 3, 4, 5\}$, there exist at most two singular moduli of discriminant $\Delta$ with denominator $a$.
		\item There exist at most four singular moduli of discriminant $\Delta$ with denominator $6$.
		\item If $\Delta \equiv 1 \bmod 8$ and $\Delta \notin \{-7, -15\}$, then there exist exactly two subdominant singular moduli of discriminant $\Delta$.
		\item If $\Delta \equiv 4 \bmod 32$, then there are no subdominant singular moduli of discriminant $\Delta$.
	\end{enumerate}
\end{prop}

The absolute value of the difference of two distinct singular moduli is bounded below in terms of the corresponding discriminants.

\begin{prop}[{\cite[Theorem~1.1]{BiluFayeZhu19}}]\label{prop:sep}
	Let $x, y$ be distinct singular moduli of respective discriminants $\Delta_x, \Delta_y$. Then
	\[ \lvert x - y \rvert \geq 800 \max \{\lvert \Delta_x \rvert, \lvert \Delta_y \rvert\}^{-4}.\]
\end{prop}

The discriminants of small class number may be bounded by an explicit computation in PARI. We will use the following bounds.

\begin{prop}\label{prop:smallclass}
	Let $\Delta < 0$ be such that $\Delta \equiv 0, 1 \bmod 4$. If $h(\Delta) \leq 6$, then $\lvert \Delta \rvert \leq 4075$. If $h(\Delta) \leq 32$, then $\lvert \Delta \rvert \leq 166147$. 
\end{prop}

\begin{proof}
	For each $k \in \{1, \ldots, 100\}$, Watkins \cite[Table~4]{Watkins04} computed the largest (in absolute value) fundamental discriminant $D$ with $h(D) = k$. As explained in \cite[Proposition~2.1]{BiluGunTron22}, one may then use this and the class number formula \cite[Theorem~7.24]{Cox22} to obtain an upper bound $c(k)$ such that $\lvert \Delta \rvert \leq c(k)$ for all $\Delta < 0$ such that $\Delta \equiv 0, 1 \bmod 4$ and $h(\Delta) \leq k$. To find the largest (in absolute value) discriminant with $h(\Delta) \leq k$, we then need only compute in PARI the class numbers of all discriminants $\Delta$ with $\lvert \Delta \rvert \leq c(k)$. This computation is essentially instant for $k=6$ and takes about $3$ minutes for $k=32$. The bounds we compute agree with those in \cite[Proposition~2.1]{BiluGunTron22} and with those given by Klaise \cite[p.19]{Klaise12}.
\end{proof}

\subsection{Fields generated by singular moduli}\label{subsec:fieldgen}

The fields generated by $\Q$-linear combinations of two singular moduli may be described explicitly.

\begin{thm}[{\cite[Theorem~1.5]{BiluFayeZhu19} \& \cite[Theorem~4.1]{FayeRiffaut18}}]\label{thm:fields}
	Let $\alpha \in \Q \setminus \{0\}$. Let $x, y$ be distinct singular moduli of respective discriminants $\Delta_x, \Delta_y$. If 
	\[ \Q(x + \alpha y) \subsetneq \Q(x, y),\]
	then one of the following holds:
	\begin{enumerate}
		\item $\alpha = 1$ and $\Delta_x = \Delta_y$. In this case, 
		\[ [\Q(x, y) : \Q(x + y)] = 2.\]
		\item $\Delta_x \neq \Delta_y$, $\Q(x) = \Q(y)$ and this field has degree $2$ over $\Q$, and
		\[ \alpha = -\frac{x - x'}{y - y'},\]
		where $x', y'$ are the non-trivial $\Q$-conjugates of $x, y$ respectively. In this case,
		\[ \Q(x + \alpha y) = \Q.\]
	\end{enumerate}
\end{thm}

We also need the following results describing when singular moduli $x, y$ satisfy either $\Q(x) = \Q(y)$ or $\Q(x) \subset \Q(y)$ with $[\Q(y) : \Q(x)] = 2$.

\begin{prop}[{\cite[Corollary~4.2]{AllombertBiluMadariaga15}}]\label{prop:difffund}
	Let $x, y$ be singular moduli with discriminants $\Delta_x, \Delta_y$ and fundamental discriminants $D_x, D_y$ respectively. If
	\[ D_x \neq D_y \mbox{ and } \Q(x) = \Q(y),\]
	then all the possible fields $\Q(x) = \Q(y)$ and the corresponding possible discriminants $\Delta_x, \Delta_y$ are listed in \cite[Table~2]{AllombertBiluMadariaga15}.
\end{prop}

\begin{prop}\label{prop:difffundsub}
	Let $x, y$ be singular moduli with discriminants $\Delta_x, \Delta_y$ and fundamental discriminants $D_x, D_y$ respectively. If
	\[ D_x \neq D_y, \, \Q(x) \subset \Q(y), \mbox{ and } [\Q(y) : \Q(x)] = 2,\] 
	then there are only finitely many possibilities for the pair $(\Delta_x, \Delta_y)$ and these may be listed explicitly.
\end{prop}

This is an explicit version of \cite[Corollary 2.13(2)]{BiluGunTron22}. To give a proof, we need the following terminology from \cite{BiluGunTron22}. Let $G$ be a finite abelian group, written multiplicatively. Say that $G$ is $2$-elementary if every element of $G$ has order $\leq 2$. Say that $G$ is almost $2$-elementary if there exists a subgroup $H \leq G$ such that $H$ is $2$-elementary and $[G : H] = 2$. Denote by $\rho_2(G)$ the dimension of $G / G^2$ as an $\mathbb{F}_2$-vector space, where $G^2 = \{g^2 : g \in G\}$.

\begin{proof}	
	Suppose that
	\begin{align}\label{eq:field}
		D_x \neq D_y, \, \Q(x) \subset \Q(y), \mbox{ and } [\Q(y) : \Q(x)] = 2.
	\end{align}
	Then $h(\Delta_x) \leq 16$ and $h(\Delta_y) \leq 32$ by \cite[Corollary~2.13]{BiluGunTron22}. So $\lvert \Delta_x \rvert, \lvert \Delta_y \rvert \leq 166147$ by Proposition~\ref{prop:smallclass}. The proof of \cite[Corollary~2.13]{BiluGunTron22} also shows that the class group $\cl(\Delta_x)$ is $2$-elementary (and hence also almost $2$-elementary) and that the class group $\cl(\Delta_y)$ is almost $2$-elementary. 
	
	One may then compute in PARI all the almost $2$-elementary discriminants $\Delta$ with $\lvert \Delta \rvert \leq 166147$ and $h(\Delta) \leq 32$ and then, among those, compute all the $2$-elementary discriminants $\Delta$ with $h(\Delta) \leq 16$. To do this, we use the following two characterisations \cite[(2.9), (2.10)]{BiluGunTron22}:
	\begin{align*} 
		\Delta \mbox{ is $2$-elementary} &\iff h(\Delta) = 2^{\rho_2(\cl (\Delta))}\\
		\Delta \mbox{ is almost $2$-elementary} &\iff h(\Delta) \mid 2^{\rho_2(\cl (\Delta))+1}
	\end{align*}
	and the classical fact (see e.g. \cite[Proposition~3.11]{Cox22} and the isomorphism in \cite[(3.19)]{Cox22}) that
	\[ \rho_2(\cl(\Delta)) = \begin{cases}
		\omega(\Delta) & \mbox{if } \Delta \equiv 0 \bmod 32\\
		\omega(\Delta)-2 & \mbox{if } \Delta \equiv 4 \bmod 16\\
		\omega(\Delta) - 1 & \mbox{otherwise,}
	\end{cases} \]
	where $\omega(\Delta)$ denotes the number of distinct prime divisors of $\Delta$ as usual. 
	
	Let $\Delta_y$ be an almost $2$-elementary discriminant such that $h(\Delta_y)\leq 32$ and $\Delta_x$ a $2$-elementary discriminant such that $h(\Delta_x) \leq 16$. One may then check in PARI whether the conditions in \eqref{eq:field} are satisfied by the pair $(\Delta_x, \Delta_y)$. One finds in this way that there are $873$ different possibilities for $(\Delta_x, \Delta_y)$.  The computations take about 30 minutes in total.
\end{proof}

\begin{prop}[{\cite[Proposition~4.3]{AllombertBiluMadariaga15} \& \cite[\S3.2.2]{BiluLucaMadariaga16}}]\label{prop:samefund}
	Let $x, y$ be singular moduli with respective discriminants $\Delta_x, \Delta_y$ and respective fundamental discriminants $D_x, D_y$. Suppose that $D_x = D_y$ and $x, y \notin \Q$. Suppose further that $\Q(x) = \Q(y)$. Then $\Delta_x/ \Delta_y \in \{1, 4, 1/4\}$ and if $\Delta_x / \Delta_y = 4$, then $\Delta_y \equiv 1 \bmod 8$.
\end{prop}

\begin{prop}[{\cite[Lemma~2.3]{Fowler20}}]\label{prop:samefundsub}
	Let $x, y$ be singular moduli with respective discriminants $\Delta_x, \Delta_y$ and respective fundamental discriminants $D_x, D_y$. Suppose that 
	\[x \notin \Q, \, D_x = D_y, \, \Q(x) \subset \Q(y), \mbox{ and } [\Q(y) : \Q(x)] = 2.\]	
	Then $\Delta_y/ \Delta_x \in \{9/4, 4, 9, 16\}$. In addition, if $\Delta_y / \Delta_x = 9/4$, then $\Delta_x \equiv 0, 4 \bmod 16$. 
\end{prop}

The ``in addition'' part of Proposition~\ref{prop:samefundsub} is not stated in \cite[Lemma~2.3]{Fowler20}, but is a trivial consequence of the fact that $\Delta_x, \Delta_y \equiv 0, 1 \bmod 4$.

\section{The proof of Theorem~\ref{thm:trip}}\label{sec:trip}

The ``if'' direction is clear. We prove the ``only if''. The proof is rather lengthy and involves considering several cases; we split the argument into a series of claims for the reader's convenience.

Let $x, y, z$ be pairwise distinct singular moduli. Suppose that 
\[A x + B y + C z \in \Q\] 
for some $A, B, C \in \Q \setminus \{0\}$. Write $\Delta_x, \Delta_y, \Delta_z$ for the discriminants and $D_x, D_y, D_z$ for the fundamental discriminants of $x, y, z$ respectively.

\begin{claim}\label{cl1}
	We may assume $x, y, z \notin \Q$.
\end{claim} 

\begin{proof}
	If at least one of $x, y, z$ is in $\Q$, then \cite[Theorem~1.2]{AllombertBiluMadariaga15} implies that either case (1) or case (2) of Theorem~\ref{thm:trip} holds. 
\end{proof}

Hence, we assume subsequently that $x, y, z \notin \Q$. We thus cannot be in either of cases (1) or (2) of Theorem~\ref{thm:trip}.

\begin{claim}\label{cl2}
	We may assume that $\Delta_x, \Delta_y, \Delta_z$ are not all equal.
\end{claim}

\begin{proof}
Suppose that $\Delta_x = \Delta_y = \Delta_z$. Write $\Delta$ for this common discriminant. By Theorem~\ref{thm:equal}, either
\[ \lvert \Delta \rvert^{1/2} \leq \frac{1}{\pi} \left(9 \log\left(4\right) -2\right) = 3.334\ldots < \sqrt{12},\]
or $h(\Delta) = 3$ and $A = B = C$. If $\lvert \Delta \rvert < 12$, then $x \in \Q$ by Lemma~\ref{lem:smallclass}, contradicting Claim~\ref{cl1}. If $h(\Delta) = 3$ and $A=B=C$, then case (4) of Theorem~\ref{thm:trip} holds.	
\end{proof}

 We may thus assume subsequently that $\Delta_x, \Delta_y, \Delta_z$ are not all equal. In particular, $x, y, z$ are not all conjugate over $\Q$ and so case (4) of Theorem~\ref{thm:trip} does not hold. 

\subsection{Controlling the fields}\label{subsec:fields}

\begin{claim}\label{cl3}
	Either:
	\[ \Q(x) = \Q(y) = \Q(z),\]
	or:
	\[ \Q(x) \subsetneq \Q(y) = \Q(z) \mbox{ and } [\Q(y) : \Q(x)] = [\Q(z) : \Q(x)] = 2.\]
\end{claim}

\begin{proof}
Observe that
\[ \Q(x)  \subset \Q(y, z).\]
The same holds permuting $x, y, z$.

Suppose then that
\[ \Q(x) \subsetneq \Q(y, z).\]
Note that
\[ \Q\left(y + \frac{C}{B}z\right) = \Q(x).\]
We apply Theorem~\ref{thm:fields} to the field 
\[\Q\left(y + \frac{C}{B} z\right).\]
Since $x \notin \Q$ by Claim~\ref{cl1}, case (2) of Theorem~\ref{thm:fields} is not possible. Thus, we must have that $B = C$, $\Delta_y = \Delta_z$, and
\[ [\Q(y, z) : \Q(x)] = 2.\]

Suppose, in addition, that
\[ \Q(y) \subsetneq \Q(x, z).\]
Then, by the same argument, we must have that $A = C$ and $\Delta_x = \Delta_z$. Thus, $\Delta_x = \Delta_y = \Delta_z$, which contradicts Claim~\ref{cl2}. Similarly, if $\Q(z) \subsetneq \Q(x, y)$.

We are thus left with the following two possibilities. Either:
\[ \Q(x) = \Q(y, z), \, \Q(y) = \Q(x, z), \mbox{ and } \Q(z) = \Q(x, y);\]
or, up to permuting $x, y, z$:
\[\Q(x) \subsetneq \Q(y, z), \, \Q(y) = \Q(x, z), \mbox{ and } \Q(z) = \Q(x, y). \]
Hence, either:
\[ \Q(x) = \Q(y) = \Q(z),\]
or:
\[ \Q(x) \subsetneq \Q(y) = \Q(z) \mbox{ and } [\Q(y) : \Q(x)] = [\Q(z) : \Q(x)] = 2. \qedhere\]
\end{proof}

\begin{claim}\label{cl4}
	If 
	\[ \Q(x) = \Q(y) = \Q(z),\]
	then we may assume that this field has degree at least $3$ over $\Q$.
\end{claim}

\begin{proof}

Suppose not. Then 
\[ \Q(x) = \Q(y) = \Q(z)\]
and, by Claim~\ref{cl1}, this field must be a degree $2$ extension of $\Q$. Write $x', y', z'$ for the unique non-trivial $\Q$-conjugates of $x, y, z$ respectively. Then
\[Ax + By + Cz = Ax' + By' + Cz'.\]
Hence,
\[A = -\frac{B(y-y')+C(z-z')}{x-x'},\]
and so we must be in case (3) of Theorem~\ref{thm:trip}.
\end{proof}

\begin{claim}\label{cl5}
	If
	\[ \Q(x) \subsetneq \Q(y) = \Q(z) \mbox{ and } [\Q(y) : \Q(x)] = [\Q(z) : \Q(x)] = 2,\]
then we may assume that
	\[ [\Q(x) : \Q] > 2.\]
\end{claim}

\begin{proof}
Suppose that
\[ \Q(x) \subsetneq \Q(y) = \Q(z) \mbox{ and } [\Q(y) : \Q(x)] = [\Q(z) : \Q(x)] = 2.\]
Suppose, in addition, that
\[ [\Q(x) : \Q] = 2.\]
We will show that we must then be in case (5) of Theorem~\ref{thm:trip}. 

Since $\Q(x) \subsetneq \Q(y, z)$, we must have that $B=C$ and $\Delta_y = \Delta_z$. Let $x'$ denote the unique non-trivial Galois conjugate of $x$ over $\Q$. Let $v, w$ be the other (along with $y, z$) two singular moduli of discriminant $\Delta_y$. Since
\[ Ax + B(y + z) \in \Q,\]
there are then $y', z' \in \{y, z, v, w\}$ such that
\[ Ax' + B(y'+z') = Ax + B(y + z).\]
So
\[ A(x-x') = -B((y+z) - (y'+z')).\]
Suppose that $\{y, z\} \cap \{y', z'\} \neq \emptyset$. Without loss of generality, we may assume that $y = y'$. So
\[ A(x-x') = B(z-z')\]
and hence
\[ \Q(x-x') = \Q(z-z').\]
We thus have that
\[ \Q(z - z') = \Q(x - x') = \Q(x, x') = \Q(x) \neq \Q,\]
where the second equality is by Theorem~\ref{thm:fields} and the third equality holds since $[\Q(x) : \Q] = 2$.
In particular, $z \neq z'$, and so, by Theorem~\ref{thm:fields} again,
\[    \Q(z - z') = \Q(z, z').\]
Hence,
\[\Q(z, z') = \Q(x).\]
This is a contradiction, since $[\Q(x) : \Q] = 2$, but
\[ [\Q(z, z') : \Q] \geq [\Q(z) : \Q] = [\Q(z) : \Q(x)] [\Q(x) : \Q] = 4.\]
Therefore, we must have that $\{y', z'\} = \{v, w\}$ and so
\[ \frac{A}{B} = - \frac{(y+z) - (v+w)}{x - x'}.\]
We are thus in case (5) of Theorem~\ref{thm:trip}. 
\end{proof}

\begin{claim}\label{cl6}
	We may assume that we are not in any of cases (1)--(5) of Theorem~\ref{thm:trip}.
\end{claim}

\begin{proof}
	By Claim~\ref{cl3}, either
		\[ \Q(x) = \Q(y) = \Q(z),\]
	or:
	\[ \Q(x) \subsetneq \Q(y) = \Q(z) \mbox{ and } [\Q(y) : \Q(x)] = [\Q(z) : \Q(x)] = 2.\]
	If 
	\[ \Q(x) = \Q(y) = \Q(z),\]
	then, by Claim~\ref{cl4}, this field has degree at least $3$ over $\Q$, and hence we are not in any of cases (1)--(5) of Theorem~\ref{thm:trip}. (Recall that, by Claim~\ref{cl2}, we are not in case (4) of Theorem~\ref{thm:trip}.)
	If 
	\[ \Q(x) \subsetneq \Q(y) = \Q(z) \mbox{ and } [\Q(y) : \Q(x)] = [\Q(z) : \Q(x)] = 2,\]
	then, by Claim~\ref{cl5},
	\[[\Q(x) : \Q] > 2,\]
	and hence we are not in any of cases (1)--(5) of Theorem~\ref{thm:trip}.
\end{proof}

\subsection{Controlling the discriminants}\label{subsec:disc}

\begin{claim}\label{cl7}
	We may assume that $D_x = D_y = D_z$.
\end{claim}

\begin{proof}
	By Claim~\ref{cl3}, either:
	\[ \Q(x) = \Q(y) = \Q(z),\]
	or:
	\[ \Q(x) \subsetneq \Q(y) = \Q(z) \mbox{ and } [\Q(y) : \Q(x)] = [\Q(z) : \Q(x)] = 2.\]
	In either case, if the fundamental discriminants $D_x, D_y, D_z$ are not all equal, then all the possibilities for $(\Delta_x, \Delta_y, \Delta_z)$ are given by Propositions~\ref{prop:difffund} and~\ref{prop:difffundsub}. These possibilities will be dealt with in Section~\ref{subsec:elim}.
\end{proof}

\begin{claim}\label{cl8}
	If $\Q(x) = \Q(y) = \Q(z)$, then we may assume that either:
	\begin{enumerate}
		\item $\Delta_x = 4 \Delta$ and $\Delta_y = \Delta_z = \Delta$ for some $\Delta \equiv 1 \bmod 8$; or
		\item $\Delta_x = \Delta_y = 4 \Delta$ and $\Delta_z = \Delta$ for some $\Delta \equiv 1 \bmod 8$.
	\end{enumerate}
\end{claim}

\begin{proof}
	Suppose that
	\[ \Q(x) = \Q(y) = \Q(z).\]
	Without loss of generality, we assume that $\lvert \Delta_x \rvert \geq \lvert \Delta_y \rvert \geq \lvert \Delta_z \rvert$. Hence, $\lvert \Delta_x \rvert > \lvert \Delta_z \rvert$ by Claim~\ref{cl2}. By Claim~\ref{cl7}, $D_x = D_y = D_z$. Hence, Proposition~\ref{prop:samefund} implies that there are the following two possibilities:
	\begin{enumerate}
		\item $\Delta_x = 4 \Delta$ and $\Delta_y = \Delta_z = \Delta$ for some $\Delta \equiv 1 \bmod 8$;
		\item $\Delta_x = \Delta_y = 4 \Delta$ and $\Delta_z = \Delta$ for some $\Delta \equiv 1 \bmod 8$ \qedhere.
	\end{enumerate}
\end{proof}

\begin{claim}\label{cl9}
	If
	\[\Q(x) \subsetneq \Q(y) = \Q(z),\]
	then $\Delta_x = \Delta$ and $\Delta_y = \Delta_z = l^2 \Delta$ for some $\Delta$, where $l \in \{3/2, 2, 3, 4\}$. If $l = 3/2$, then $\Delta \equiv 0, 4 \bmod 16$.
\end{claim}

\begin{proof}
	Since
	\[ \Q(B y + C z) = \Q(x) \subsetneq \Q(y) = \Q(z),\]
Theorem~\ref{thm:fields} and Claim~\ref{cl1} together imply that $\Delta_y = \Delta_z$ and $B=C$. Proposition~\ref{prop:samefundsub} then implies that we may write $\Delta_x = \Delta$ and $\Delta_y = \Delta_z = l^2 \Delta$, where $l \in \{3/2, 2, 3, 4\}$, and also that $\Delta \equiv 0, 4 \bmod 16$ if $l = 3/2$.
\end{proof}

\subsection{Bounding the discriminants}\label{subsec:bd}

Let
\[ \alpha = A x + B y + C z.\]
So $\alpha \in \Q$ by assumption. Suppose that $(x_i, y_i, z_i)$, where $i=1, \ldots, 4$, are Galois conjugates of $(x, y, z)$ over $\Q$. Then
\[ A x_i + B y_i + C z_i = \alpha\]
for $i=1, 2, 3, 4$. In particular, we have that
\begin{align}\label{eq:det}
	\begin{vmatrix}
		1 & 1 & 1 & 1\\
		x_1 & x_2 & x_3 & x_4\\
		y_1 & y_2 & y_3 & y_4\\
		z_1 & z_2 & z_3 & z_4
	\end{vmatrix}= 0.
\end{align}

In this section, we will establish effective bounds on $(\Delta_x, \Delta_y, \Delta_z)$ by showing that, if these discriminants are too large (in absolute value), then one may find conjugates $(x_i, y_i, z_i)$ such that the determinant on the left hand side of \eqref{eq:det} is non-zero. Typically, we will do this by showing that there exists a permutation $\tau \in S_4$ such that
\begin{align*}
	\lvert x_{\tau(2)} y_{\tau(3)} z_{\tau(4)} \rvert > \sum_{\sigma \in S_4 \setminus \{\tau\}} \lvert x_{\sigma(2)} y_{\sigma(3)} z_{\sigma(4)} \rvert.
\end{align*}
We will make repeated use (without special reference) of the bound on singular moduli given in Proposition~\ref{prop:bd}, the description of the conjugates of a singular modulus in Subsection~\ref{subsec:singmods}, and the bound on the number of singular moduli of a given discriminant and denominator in Proposition~\ref{prop:denom}. The bounds themselves were calculated in PARI; this took only a few seconds.

\subsubsection{The case where $\Q(x) = \Q(y) = \Q(z)$}\label{subsubsec:equalfielddisc}

In this section, we assume that
\[ \Q(x) = \Q(y) = \Q(z).\]
By Claim~\ref{cl4}, the field $\Q(x) = \Q(y) = \Q(z)$ has degree at least $3$ over $\Q$. Each singular modulus of discriminant $\Delta_x$ occurs precisely once among the first coordinates of the distinct $\Q$-conjugates of $(x, y, z)$. Correspondingly for the singular moduli of discriminant $\Delta_y$ among the second coordinates, and for the singular moduli of discriminant $\Delta_z$ among the third coordinates.

\begin{claim}\label{cl10}
	If $\Delta_x = 4 \Delta$ and $\Delta_y = \Delta_z = \Delta$, where $\Delta \equiv 1 \bmod 8$, then $h(\Delta) \leq 6$.
\end{claim}

\begin{proof}
Suppose that $\Delta_x = 4 \Delta$ and $\Delta_y = \Delta_z = \Delta$, where $\Delta \equiv 1 \bmod 8$. Assume that $h(\Delta) \geq 7$. In particular, $\Delta \notin \{-7, -15\}$ by Lemma~\ref{lem:smallclass}. Thus there are exactly two subdominant singular moduli of discriminant $\Delta$ and no subdominant singular moduli of discriminant $4 \Delta$. Let $(x_1, y_1, z_1)$ be the unique conjugate of $(x, y, z)$ with $x_1$ dominant. We distinguish two subcases.

\textit{Subcase 1: neither of $y_1, z_1$ is dominant.} In this case, we may take $(x_2, y_2, z_2), (x_3, y_3, z_3)$ to be the unique conjugates such that $y_2, z_3$ are dominant. Then
\begin{align}\label{bd3}
	\lvert x_1 y_2 z_3 \rvert \geq (\exp(2 \pi \lvert \Delta \rvert^{1/2}) - 2079)(\exp(\pi \lvert \Delta \rvert^{1/2}) - 2079)^2.
\end{align}
Since $h(\Delta) \geq 7$, we may take $(x_4, y_4, z_4)$ to be a conjugate with none of $x_4, y_4, z_4$ dominant. Then, for $\sigma \in S_4$ such that $\sigma(2, 3, 4) \neq (1, 2, 3)$, we have that either: $\sigma(2)=1$ and
\begin{align}\label{bd4}
	&\lvert x_{\sigma(2)} y_{\sigma(3)} z_{\sigma(4)} \rvert \nonumber \\ 
	\leq &(\exp(2 \pi \lvert \Delta \rvert^{1/2}) + 2079)(\exp(\pi \lvert \Delta \rvert^{1/2}) + 2079)\nonumber \\ 
	&\left(\exp\left(\frac{\pi \lvert \Delta \rvert^{1/2}}{2}\right) + 2079\right),
\end{align}
or: $\sigma(2) \neq 1$ and
\begin{align}\label{bd5}
	&\lvert x_{\sigma(2)} y_{\sigma(3)} z_{\sigma(4)} \rvert \nonumber \\
	\leq &\left(\exp\left(\frac{2 \pi \lvert \Delta \rvert^{1/2}}{3}\right) + 2079\right)(\exp(\pi \lvert \Delta \rvert^{1/2}) + 2079)^2.
\end{align}
The bounds \eqref{bd3}, \eqref{bd4}, and \eqref{bd5} are incompatible with \eqref{eq:det} if $\lvert \Delta \rvert \geq 10$. By Lemma~\ref{lem:smallclass}, $\lvert \Delta \rvert \leq 9$ is impossible since $z \notin \Q$.

\textit{Subcase 2: one of $y_1, z_1$ is dominant.} Without loss of generality, suppose that $y_1$ is dominant. Let $(x_2, y_2, z_2)$ be the unique conjugate with $z_2$ dominant. There also exists a conjugate $(x_3, y_3, z_3)$ with $y_3$ subdominant and $x_3, z_3$ not dominant. There exists a conjugate $(x_4, y_4, z_4)$ with neither of $y_4, z_4$ dominant or subdominant. Note that $x_4$ is thus not dominant as well. Hence,
\begin{align}\label{bd6}
	&\lvert x_1 y_3 z_2 \rvert \nonumber\\
	\geq &(\exp(2 \pi \lvert \Delta \rvert^{1/2}) - 2079)\left(\exp\left(\frac{\pi \lvert \Delta \rvert^{1/2}}{2}\right) - 2079\right) \nonumber \\
	&(\exp(\pi \lvert \Delta \rvert^{1/2}) - 2079).
\end{align}
Now let $\sigma \in S_4$ be such that $\sigma(2, 3, 4) \neq (1, 3, 2)$. If $\sigma(2) \neq 1$, then
\begin{align}\label{bd7}
	&\lvert x_{\sigma(2)} y_{\sigma(3)} z_{\sigma(4)} \rvert \nonumber \\
	\leq &\left(\exp\left(\frac{2 \pi \lvert \Delta \rvert^{1/2}}{3}\right) + 2079\right)(\exp(\pi \lvert \Delta \rvert^{1/2}) + 2079)^2.
\end{align}
If $\sigma(2) = 1$ and $\sigma(4) = 2$, then
\begin{align}\label{bd8}
	&\lvert x_{\sigma(2)} y_{\sigma(3)} z_{\sigma(4)} \rvert \nonumber \\ 
	\leq &(\exp(2 \pi \lvert \Delta \rvert^{1/2}) + 2079)\left(\exp\left(\frac{\pi \lvert \Delta \rvert^{1/2}}{3}\right) + 2079\right) \nonumber \\
	&(\exp(\pi \lvert \Delta \rvert^{1/2}) + 2079).
\end{align}
If $\sigma(2) = 1$ and $\sigma(4) \neq 2$, then
\begin{align}\label{bd9}
	&\lvert x_{\sigma(2)} y_{\sigma(3)} z_{\sigma(4)} \rvert \nonumber \\
	\leq &(\exp(2 \pi \lvert \Delta \rvert^{1/2}) + 2079)\left(\exp\left(\frac{\pi \lvert \Delta \rvert^{1/2}}{2}\right) + 2079\right)^2.
\end{align}
The bounds \eqref{bd6}, \eqref{bd7}, \eqref{bd8}, and \eqref{bd9} are incompatible with \eqref{eq:det} if $\lvert \Delta \rvert \geq 32$. By Lemma~\ref{lem:smallclass}, $\lvert \Delta \rvert \leq 31$ is impossible since $h(\Delta) \geq 7$ by assumption.

We must therefore have that $h(\Delta) \leq 6$.
\end{proof}

\begin{claim}\label{cl11}
	If $\Delta_x = \Delta_y = 4 \Delta$ and $\Delta_z = \Delta$, where $\Delta \equiv 1 \bmod 8$, then $h(\Delta) \leq 5$.
\end{claim}

\begin{proof}
	Suppose that $\Delta_x = \Delta_y = 4 \Delta$ and $\Delta_z = \Delta$, where $\Delta \equiv 1 \bmod 8$. Assume further that $h(\Delta) \geq 6$. In particular, $\Delta \notin \{-7, -15\}$ by Lemma~\ref{lem:smallclass}. Thus there are exactly two subdominant singular moduli of discriminant $\Delta$ and no subdominant singular moduli of discriminant $4 \Delta$. Let $(x_1, y_1, z_1), (x_2, y_2, z_2)$ be the unique conjugates of $(x, y, z)$ such that $x_1, y_2$ are dominant. There must also exist a conjugate $(x_3, y_3, z_3)$ such that $z_3$ is either dominant or subdominant and neither of $x_3, y_3$ is dominant.
	
	Then 
	\begin{align}\label{bd10}
		\lvert x_1 y_2 z_3 \rvert \geq (\exp(2 \pi \lvert \Delta \rvert^{1/2}) - 2079)^2\left(\exp\left(\frac{\pi \lvert \Delta \rvert^{1/2}}{2}\right) - 2079\right).
	\end{align}
There exists a conjugate $(x_4, y_4, z_4)$ such that neither of $x_4, y_4$ is dominant and $z_4$ is neither dominant nor subdominant. Let $\sigma \in S_4$ be such that $\sigma(2, 3, 4) \neq (1, 2, 3)$. If $\sigma(2)=1$ and $\sigma(3)=2$, then
	\begin{align}\label{bd11}
		&\lvert x_{\sigma(2)} y_{\sigma(3)} z_{\sigma(4)} \rvert \nonumber \\ 
		\leq &(\exp(2 \pi \lvert \Delta \rvert^{1/2}) + 2079)^2\left(\exp\left(\frac{\pi \lvert \Delta \rvert^{1/2}}{3}\right) + 2079\right).
	\end{align}
	If either $\sigma(2) \neq 1$ or $\sigma(3) \neq 2$, then
	\begin{align}\label{bd12}
		&\lvert x_{\sigma(2)} y_{\sigma(3)} z_{\sigma(4)} \rvert \nonumber \\
		\leq &(\exp(2 \pi \lvert \Delta \rvert^{1/2}) + 2079)\left(\exp\left(\frac{2 \pi \lvert \Delta \rvert^{1/2}}{3}\right) + 2079\right) \nonumber \\
		&(\exp(\pi \lvert \Delta \rvert^{1/2}) + 2079).
	\end{align}
	The bounds \eqref{bd10}, \eqref{bd11}, and \eqref{bd12} are incompatible with \eqref{eq:det} if $\lvert \Delta \rvert \geq 29$. By Lemma~\ref{lem:smallclass}, there are no discriminants $\Delta$ with $\lvert \Delta \rvert \leq 28$ and $h(\Delta) \geq 6$. Hence, we must have that $h(\Delta) \leq 5$.
\end{proof}

\subsubsection{The case where $\Q(x) \subsetneq \Q(y) = \Q(z)$}\label{subsubsec:subfielddisc}

In this section, we assume that
\[ \Q(x) \subsetneq \Q(y) = \Q(z).\]
By Claim~\ref{cl9}, we have that $\Delta_x = \Delta$ and $\Delta_y = \Delta_z = l^2 \Delta$, where $l \in \{3/2, 2, 3, 4\}$. Also, by Claims~\ref{cl3} and \ref{cl5},
\[ [\Q(y) : \Q(x)] = [\Q(z) : \Q(x)] = 2 \mbox{ and } [\Q(x) : \Q] \geq 3.\]
So $h(\Delta) \geq 3$.

Each singular modulus of discriminant $\Delta$ occurs exactly twice among the first coordinates of the distinct $\Q$-conjugates of $(x, y, z)$. Each singular modulus of discriminant $l^2 \Delta$ occurs exactly once among the second coordinates of the distinct $\Q$-conjugates of $(x, y, z)$ and exactly once among the third coordinates. Let $(x_1, y_1, z_1), (x_2, y_2, z_2)$ be the conjugates such that $y_1, z_2$ are dominant.

\begin{claim}\label{cl12}
	Both of $x_1, x_2$ are dominant.
\end{claim}

\begin{proof}
	Suppose that at least one of $x_1, x_2$ is not dominant. There is then a conjugate $(x_3, y_3, z_3)$ with $x_3$ dominant and neither of $y_3, z_3$ dominant. Then
	\begin{align}\label{bd13}
		\lvert x_3 y_1 z_2 \rvert \geq (\exp(\pi \lvert \Delta \rvert^{1/2}) - 2079)(\exp(l \pi \lvert \Delta \rvert^{1/2}) - 2079)^2.
	\end{align}
	Since $h(\Delta) \geq 3$, there is a conjugate $(x_4, y_4, z_4)$ with none of $x_4, y_4, z_4$ dominant. Now let $\sigma \in S_4$ be such that $\sigma(2, 3, 4) \neq (3, 1, 2)$. If $\sigma(3) = 1$ and $\sigma(4) = 2$, then
	\begin{align}\label{bd14}
		&\lvert x_{\sigma(2)} y_{\sigma(3)} z_{\sigma(4)} \rvert \nonumber \\
		\leq &\left(\exp\left(\frac{\pi \lvert \Delta \rvert^{1/2}}{2}\right) + 2079\right)(\exp(l \pi \lvert \Delta \rvert^{1/2}) + 2079)^2.
	\end{align}
	If either $\sigma(3) \neq 1$ or $\sigma(4) \neq 2$, then
	\begin{align}\label{bd15}
		&\lvert x_{\sigma(2)} y_{\sigma(3)} z_{\sigma(4)} \rvert \nonumber \\ 
		\leq &(\exp(\pi \lvert \Delta \rvert^{1/2}) + 2079)(\exp(l \pi \lvert \Delta \rvert^{1/2}) + 2079) \nonumber \\
		&\left(\exp\left(\frac{l \pi \lvert \Delta \rvert^{1/2}}{2}\right) + 2079\right).
	\end{align}
	For every $l \in \{3/2, 2, 3, 4\}$, the bounds \eqref{bd13}, \eqref{bd14}, and \eqref{bd15} are incompatible with \eqref{eq:det} if $\lvert \Delta \rvert \geq 8$. By Lemma~\ref{lem:smallclass}, $\lvert \Delta \rvert \leq 7$ is impossible since $x \notin \Q$.
\end{proof}

Note that $x_1 = x_2$, since $x_1, x_2$ are both dominant. We therefore rearrange \eqref{eq:det} to obtain that
\begin{align}\label{eq:det2}
	y_1 z_2 (x_3 - x_4)  = \sum_{\substack{\sigma \in S_4\\ \sigma(3, 4) \neq (1, 2)}}  \mathrm{sgn}(\sigma) x_{\sigma(2)} y_{\sigma(3)} z_{\sigma(4)}.
\end{align}
We may then use Proposition~\ref{prop:sep} to obtain the lower bound
\begin{align}\label{bd16}
	\lvert y_1 z_2 (x_3 - x_4) \rvert \geq 800 \frac{(\exp(l \pi \lvert \Delta \rvert^{1/2}) - 2079)^2}{\lvert \Delta \rvert^{4}}.
\end{align}

\begin{claim}\label{cl13}
	We have that $l \in \{3/2, 2\}$ and:
	\begin{enumerate}
		\item If $l = 3/2$, then one of the following holds:
		\begin{enumerate}
			\item $h(\Delta_x) \leq 6$;
			\item $h(\Delta_x) \in \{7, 8\}$ and $\lvert \Delta_x \rvert \leq 5879$;
			\item $h(\Delta_x) \in \{9, 10\}$ and $\lvert \Delta_x \rvert \leq 1557$;
			\item $h(\Delta_x) \in \{11, 12, 13, 14\}$ and $\lvert \Delta_x \rvert \leq 790$.
		\end{enumerate}
		\item If $l = 2$, then one of the following holds:
		\begin{enumerate}
			\item $h(\Delta_x) \leq 4$;
			\item $h(\Delta_x) \in \{5, 6\}$ and $\lvert \Delta_x \rvert \leq 304$;
		\end{enumerate}
	\end{enumerate}
\end{claim}

\begin{proof}

For each $k \in \{3, 5, 7, 9, 11, 15\}$, there is some $a_\mathrm{min}(k)$, the value of which is given in Table~\ref{tab:bds}, with the property that: for every $l \in \{3/2, 2, 3, 4\}$, if $h(\Delta) \geq k$, then there exist conjugates $(x_3, y_3, z_3), (x_4, y_4, z_4)$ such that $x_3 \neq x_4$ and
\[ a(y_3), a(y_4), a(z_3), a(z_4) \geq a_\mathrm{min}(k),\]
where $a(y_i)$ (respectively $a(z_i)$) denotes the denominator of $y_i$ (respectively $z_i$).

\begin{table}[ht]
	\centering
	\begin{tabular}{|c |c|c|c|c|c|}
		\hline
		\multirow{2}{*}{$k$} & \multirow{2}{*}{$a_\mathrm{min}(k)$} & \multicolumn{4}{c|}{$l$} \\
		\cline{3-6}
		& &	$3/2$ & $2$ & $3$ & $4$\\
		\hline
		$3$ & 2 & & & \cellcolor[gray]{0.9} 2 & \cellcolor[gray]{0.9} 0\\
		$5$ & 3 & & 304 & &\\
		$7$ & 4 & 5879 & \cellcolor[gray]{0.9} 49 & &\\
		$9$ & 5 & 1557 &  & &\\
		$11$ & 6 & 790 & & &\\
		$15$ & 7 & \cellcolor[gray]{0.9} 515 & & &\\
		\hline
	\end{tabular}
	\caption{Upper bounds on $\lvert \Delta \rvert$ when $h(\Delta) \geq k$ for a given $l \in \{3/2, 2, 3, 4\}$.}\label{tab:bds}
\end{table}

Fix some $l \in \{3/2, 2, 3, 4\}$. Next, let $k \in \{3, 5, 7, 9, 11, 15\}$, and take the corresponding value of $a_\mathrm{min}(k)$ from Table~\ref{tab:bds}. Now let $\sigma \in S_4$ be such that $\sigma(3, 4) \neq (1, 2)$. If either $\sigma(3) = 1$ or $\sigma(4) = 2$ (note there are $8$ such $\sigma$), then 
\begin{align}\label{bd17}
	&\lvert x_{\sigma(2)} y_{\sigma(3)} z_{\sigma(4)} \rvert \nonumber \\ 
	\leq &(\exp(\pi \lvert \Delta \rvert^{1/2}) + 2079)(\exp(l \pi \lvert \Delta \rvert^{1/2}) + 2079) \nonumber \\
	&\left(\exp\left(\frac{l \pi \lvert \Delta \rvert^{1/2}}{a_\mathrm{min}(k)}\right) + 2079\right).
\end{align}
If $\sigma(3) = 2$ and $\sigma(4) = 1$ (note there are $2$ such $\sigma$), then 
\begin{align}\label{bd18}
	&\lvert x_{\sigma(2)} y_{\sigma(3)} z_{\sigma(4)} \rvert \nonumber \\
	\leq &\left(\exp\left(\frac{\pi \lvert \Delta \rvert^{1/2}}{2}\right) + 2079\right)\left(\exp\left(\frac{l \pi \lvert \Delta \rvert^{1/2}}{2}\right) + 2079\right)^2.
\end{align}
Otherwise (i.e. for the $12$ remaining $\sigma$),
\begin{align}\label{bd19}
	&\lvert x_{\sigma(2)} y_{\sigma(3)} z_{\sigma(4)} \rvert \nonumber \\
	\leq &(\exp(\pi \lvert \Delta \rvert^{1/2}) + 2079)\left(\exp\left(\frac{l \pi \lvert \Delta \rvert^{1/2}}{2}\right) + 2079\right) \nonumber \\
	&\left(\exp\left(\frac{l \pi \lvert \Delta \rvert^{1/2}}{a_\mathrm{min}(k)}\right) + 2079\right).
\end{align}

If $k$ is suitably large (for the given $l$), then the bounds \eqref{bd16}, \eqref{bd17}, \eqref{bd18}, and \eqref{bd19} are incompatible with \eqref{eq:det2} for large enough $\lvert \Delta \rvert$. The entries in Table~\ref{tab:bds} are the upper bounds on $\lvert \Delta \rvert$ which are obtained for a given $l$ when $h(\Delta) \geq k$ by taking suitable conjugates $(x_3, y_3, z_3), (x_4, y_4, z_4)$ with $x_3 \neq x_4$ and
\[ a(y_3), a(y_4), a(z_3), a(z_4) \geq a_\mathrm{min}(k).\]

Note that for $l = 3/2$ (respectively, for $l=2$) no upper bound on $\lvert \Delta \rvert$ is obtained if $k < 7$ (respectively if $k < 5$). An entry $b$ in a row $k$ and column $l$ of Table~\ref{tab:bds} is in a grey-shaded cell if there are no discriminants $\Delta$ satisfying:
\begin{enumerate}
	\item $h(\Delta) \geq k$,
	\item $\lvert \Delta \rvert \leq b$, and
	\item if $l=3/2$, then $\Delta \equiv 0, 4 \bmod 16$.
\end{enumerate}
(These conditions are easily checked in PARI.) In particular, this implies that $l \in \{3, 4\}$ is impossible, since $h(\Delta) \geq 3$ by Claim~\ref{cl5}. For $l \in \{3/2, 2\}$, we must be in one of the cases stated in the claim.
\end{proof}

\subsection{Eliminating the discriminants}\label{subsec:elim}

So far in Section~\ref{sec:trip}, we have shown that if $x, y, z$ are pairwise distinct singular moduli of respective discriminants $\Delta_x, \Delta_y, \Delta_z$ such that
\[ Ax + By + Cz \in \Q\]
for some $A, B, C \in \Q \setminus \{0\}$, then either we are in one of cases (1)--(5) of Theorem~\ref{thm:trip} or (without loss of generality) one of the following situations occurs:

\begin{enumerate}
	\item $\Q(x) = \Q(y) = \Q(z)$ and this field, denote it $L$, is an extension of $\Q$ of degree at least $3$ and one of the following holds:
	\begin{enumerate}
		\item $(\Delta_x, \Delta_y, \Delta_z) = (4 \Delta, \Delta, \Delta)$, where $\Delta \equiv 1 \bmod 8$ and $h(\Delta) \leq 6$;
		\item $(\Delta_x, \Delta_y, \Delta_z) = (4 \Delta, 4 \Delta, \Delta)$, where $\Delta \equiv 1 \bmod 8$ and $h(\Delta) \leq 5$;
		\item the field $L$ and the discriminants $\Delta_x, \Delta_y, \Delta_z$ are given in \cite[Table~2]{AllombertBiluMadariaga15};
	\end{enumerate}
	\item $\Q(x) \subset \Q(y) = \Q(z)$,  $[\Q(y) : \Q(x)] = 2$, $[\Q(x) : \Q] \geq 3$, $\Delta_y = \Delta_z$, and one of the following holds:
	\begin{enumerate}
		\item $9 \Delta_x = 4 \Delta_y$ and $\Delta_x \equiv 0, 4 \bmod 16$ and one of the following holds:
		\begin{enumerate}
			\item $h(\Delta_x) \leq 6$;
			\item $h(\Delta_x) \in \{7, 8\}$ and $\lvert \Delta_x \rvert \leq 5879$;
			\item $h(\Delta_x) \in \{9, 10\}$ and $\lvert \Delta_x \rvert \leq 1557$;
			\item $h(\Delta_x) \in \{11, 12, 13, 14\}$ and $\lvert \Delta_x \rvert \leq 790$;
		\end{enumerate}
		\item $4 \Delta_x = \Delta_y$ and one of the following holds:
		\begin{enumerate}
			\item $h(\Delta_x) \leq 4$;
			\item $h(\Delta_x) \in \{5, 6\}$ and $\lvert \Delta_x \rvert \leq 304$;
		\end{enumerate}
		\item $(\Delta_x, \Delta_y)$ is one of the 873 pairs given in Proposition~\ref{prop:difffundsub}.
	\end{enumerate}
\end{enumerate}
Here (1)(a) arises from Claim~\ref{cl10}, (1)(b) arises from Claim~\ref{cl11}, 1(c) and (2)(c) arise from Claim~\ref{cl7}, and (2)(a) and (2)(b) arise from Claim~\ref{cl13}.

The corresponding list of all such $(\Delta_x, \Delta_y, \Delta_z)$ may be generated straightforwardly in PARI. Here we make use of Proposition~\ref{prop:smallclass}, whence $h(\Delta) \leq 6$ implies that $\lvert \Delta \rvert \leq 4075$. For each triple of discriminants $(\Delta_x, \Delta_y, \Delta_z)$ belonging to this list, we then use PARI to show that there are no pairwise distinct singular moduli $x, y, z$ of the corresponding discriminants such that the set $\{1, x, y, z\}$ is linearly dependent over $\Q$. The necessary computations take about 40 minutes to run. This completes the proof of Theorem~\ref{thm:trip}. 

\begin{remark}\label{rmk:deg3}
	One might expect that, in analogy to case (3) of Theorem~\ref{thm:trip}, there would be an additional case in Theorem~\ref{thm:trip}, namely one where $x, y, z$ all generate the same number field of degree $3$ over $\Q$. In fact, this cannot occur, for reasons that we now explain.
	
	Suppose that $x, y$ are distinct singular moduli such that $\Q(x) = \Q(y)$ and this field has degree $3$ over $\Q$. In particular, by \cite[Corollary~3.3]{AllombertBiluMadariaga15}, the field $\Q(x)$ is not a Galois extension of $\Q$. Propositions~\ref{prop:difffund} and~\ref{prop:samefund}, together with the list of discriminants $\Delta$ with  $h(\Delta) = 3$, imply that either $\Delta_x = \Delta_y$ or $\{\Delta_x, \Delta_y\} \in \{ \{-23, -92\}, \{-31, -124\} \}$. 
	
	Hence, if $x, y, z$ are pairwise distinct singular moduli such that $\Q(x) = \Q(y) = \Q(z)$ and this field has degree $3$ over $\Q$, then at least two of $x, y, z$ have the same discriminant. Suppose that $x, y$ have the same discriminant. Then the equality $\Q(x) = \Q(y)$ implies that $\Q(x)$ is a Galois extension of $\Q$, which is a contradiction.
\end{remark}

\end{document}